%% file: Main.tex
\documentclass{amsart}

\input{Preamble.tex}

\usepackage{etoolbox}
\newtoggle{draft}
\togglefalse{draft}

\iftoggle{draft} {
\usepackage[margin=1.5in]{geometry}
\newcommand{\NB}[1]{\todo[color=gray!40]{#1}}
\newcommand{\TODO}[1]{\todo[color=red]{#1}}
\usepackage{showkeys}
}{ 
\usepackage[margin=1.5in]{geometry}
\newcommand{\NB}[1]{}
\newcommand{\TODO}[1]{}
\renewcommand{\todo}[1]{}
\renewcommand{\todo}[1]{}
}

\makeatletter
\let\@wraptoccontribs\wraptoccontribs
\makeatother

\title{How Big are the Stable Homotopy Groups of Spheres?}
\date{\today}

\author[Robert Burklund]{Robert Burklund,\\ \MakeLowercase{with an appendix joint with} Andrew Senger}
\address{Department of Mathematics, MIT, Cambridge, MA, USA}
\email{burklund@mit.edu}

\address{Department of Mathematics, Harvard University, Cambridge, MA, USA}
\email{senger@math.harvard.edu}

\begin{document}


  

\begin{abstract}
  The stable homotopy groups of spheres are a well-studied,
  but poorly understood repository of homotopical information.
  Often, they are considered in some highly structured sense.
  In this article we take the opposite tack, asking the titular question.
  From this viewpoint, spectral sequences become upper bounds
  and constructions, such as the Greek letter elements,
  become simple lower bounds.

  We show that
  the $p$-torsion exponent of the stable stems grows sublinearly in $n$
  and
  the $p$-rank of the $E_2$-page of the Adams spectral sequence
  grows as $\exp(\Theta( \log(n)^3))$.
  Together these bounds provide the first subexponential bound on the size of the stable stems.
  In the other direction we prove that a certain, precise, version of the failure of the telescope conjecture would imply that the upper bound provided by the Adams $E_2$-page is essentially sharp---answering the titular question:
    
  


  \vspace{0.1cm}
  \begin{quote} \centering
    As big as the fate of the telescope conjecture demands.
  \end{quote}
  \vspace{0.1cm}
  
  In an appendix joint with Andrew Senger, we consider the unstable analog of this question. Bootstrapping from the stable case, we prove that the size of the $p$-local homotopy groups of spheres is bounded by $\exp(O(\log(n)^3))$, providing the first subexponential bound on the unstable stems. 
  
\end{abstract}

\maketitle

\setcounter{tocdepth}{1}
\tableofcontents
\vbadness 5000


\section{Introduction}
\label{sec:intro}
\input{introduction.tex}

\section{$p$-torsion exponents are sublinear}
\label{sec:torexp}
\input{torexp.tex}

\section{Valuations on stable categories}
\label{sec:val}
\input{val.tex}

\section{The cohomology of the Steenrod algebra}
\label{sec:steenrod}
\input{steenrod.tex}

\section{The cohomology of the moduli of formal groups}
\label{sec:mfg}
\input{mfg.tex}

\section{Motivic stable stems over $\C$}
\label{sec:motivic}
\input{motivic.tex}

\section{The telescope conjecture and stable stems}
\label{sec:telescope}
\input{conj.tex}

\appendix

\section{Unstable homotopy groups (joint with Andrew Senger)}
\label{sec:unstable}
\input{unstable.tex}

\section{Combinatorics}
\label{app:combo}
\input{combo.tex}

\bibliographystyle{alpha}
\bibliography{bibliography}

\end{document}

%% file: Preamble.tex
\usepackage{verbatim}
\usepackage[textsize=scriptsize]{todonotes}
\usepackage{tikz-cd}
\usepackage{etoolbox}
\usepackage{etex}
\usepackage[T1]{fontenc}
\usepackage{chemarr}
\usepackage{amssymb}
\usepackage{amsmath}
\usepackage{comment}
\usepackage{mathtools}
\usepackage{rotating}
\usepackage{wrapfig}
\usepackage{outlines}
\usepackage{graphicx}
\usepackage{scalerel}
\usepackage{bbm}
\usepackage{multicol}

\usepackage{amsthm}

\usepackage{spectralsequences}

\let\oldwidetilde\widetilde
\protected\def\widetilde{\oldwidetilde}

\usepackage{hyperref}
\usepackage{cleveref}

\hypersetup{
   colorlinks,
   linkcolor={red},
   citecolor={green!50!black},
   urlcolor={blue}
}

\usepackage{tikz}
\usetikzlibrary{matrix,arrows,decorations}
\usepackage{tikz-cd}

\usepackage{adjustbox}

\let\oldtocsection=\tocsection
 
\let\oldtocsubsection=\tocsubsection
 
\let\oldtocsubsubsection=\tocsubsubsection
 
\renewcommand{\tocsection}[2]{\hspace{0em}\oldtocsection{#1}{#2}}
\renewcommand{\tocsubsection}[2]{\hspace{1em}\oldtocsubsection{#1}{#2}}
\renewcommand{\tocsubsubsection}[2]{\hspace{2em}\oldtocsubsubsection{#1}{#2}}

\theoremstyle{definition}

\newtheorem{nul}{}[section]
\newtheorem{dfn}[nul]{Definition}

\newtheorem{rmk}[nul]{Remark}

\newtheorem{cnstr}[nul]{Construction}

\newtheorem{ntn}[nul]{Notation}
\newtheorem{exm}[nul]{Example}

\newtheorem{rec}[nul]{Recollection}

\newtheorem{wrn}[nul]{Warning}
\newtheorem{qst}[nul]{Question}

\newtheorem*{dfn*}{Definition}
\newtheorem*{axm*}{Axiom}
\newtheorem*{ntn*}{Notation}
\newtheorem*{exm*}{Example}
\newtheorem*{exr*}{Exercise}
\newtheorem*{int*}{Intuition}
\newtheorem*{qst*}{Question}
\newtheorem*{rmk*}{Remark}

\theoremstyle{plain}

\newtheorem{challenge}[nul]{Challenge}

\newtheorem{thm}[nul]{Theorem}
\newtheorem{prop}[nul]{Proposition}

\newtheorem{lem}[nul]{Lemma}

\newtheorem{cnj}[nul]{Conjecture}
\newtheorem{cor}[nul]{Corollary}

\newtheorem*{thm*}{Theorem}
\newtheorem*{prop*}{Proposition}
\newtheorem*{cor*}{Corollary}
\newtheorem*{lem*}{Lemma}
\newtheorem*{cnj*}{Conjecture}

\DeclareMathOperator*{\colim}{colim}

\DeclareMathOperator{\Cof}{Cof}

\DeclareMathOperator{\Map}{Map}

\DeclareMathOperator{\Ext}{Ext}

\DeclareMathOperator{\Spec}{\mathrm{Spec}}

\DeclareMathOperator{\Mod}{Mod}

\def\H{\mathrm{H}}
\def\h{\mathrm{h}}

\def\tors{\mathrm{tors}}

\def\varpitel{\varpi_{\scaleto{\mathrm{T}(h)}{6pt}}}
\def\varpimot{\varpi_{\scaleto{\mathrm{mot}}{4.2pt}}^c}
\def\varpifg{\varpi_{\mathrm{fg}}}

\def\A{\mathbb{A}}
\def\C{\mathbb{C}}
\def\E{\mathbb{E}}
\def\F{\mathbb{F}}
\def\G{\mathbb{G}}
\def\N{\mathbb{N}}

\def\Q{\mathbb{Q}}

\def\Ss{\mathbb{S}}
\def\Z{\mathbb{Z}}

\def\CC{\mathcal{C}}
\def\DD{\mathcal{D}}
\def\AA{\mathcal{A}}

\def\Sq{\mathrm{Sq}}

\def\BP{\mathrm{BP}}

\def\Mfg{\mathcal{M}_{\mathrm{fg}}}

\def\ANE{{}^{\scaleto{\mathrm{AN}}{4pt}}\!E}
\def\AE{{}^{\scaleto{\mathrm{A}}{4pt}}\!E}
\def\gbp{g_{\scaleto{\BP}{3.9pt}}}

\def\U{\mathcal{U}}
\def\V{\mathcal{V}}
\def\rank{\mathrm{rank}}
\def\A{\mathcal{A}}
\def\CU{\mathrm{CU}}
\def\CA{\mathcal{CA}}

\def\Sp{\mathrm{Sp}}

\newcommand{\SHC}{\mathcal{SH}(\mathbb{C})}

\def\Id{\mathrm{Id}}

\newcommand{\pullback}{\arrow[dr, phantom, "\lrcorner", very near start]}

\newcommand\xqed[1]{%
  \leavevmode\unskip\penalty9999 \hbox{}\nobreak\hfill
  \quad\hbox{#1}}
\newcommand\tqed{\xqed{$\triangleleft$}}


%% file: introduction.tex
Serre's finiteness theorem tells us that, aside from $\pi_0$, the stable homotopy groups of spheres are finite \cite{SerreFinite}.
Despite substantial subsequent effort into understanding the structure of the stable homotopy category,
little progress has been made towards accurately estimating the size of these finite groups.
In this article we focus on two measures of size
\[ \tors_p(\pi_n\Ss) \coloneqq \min_k \{ k\ |\ p^{k} \cdot \pi_n\Ss_{(p)} = 0 \} \quad\text{ and }\quad \rank_p(\pi_n\Ss) \]
and begin the process of isolating exactly what is at stake in this question.




\subsection{Torsion bounds}\ 

Torsion bounds on homotopy groups have a rich history in the unstable setting,
beginning with work of James and Toda \cite{JTor, TTor} and 
culminating in the celebrated Cohen--Moore--Niesendorfer theorem \cite{CMN1,CMN2,N3}
which provides optimal, \emph{uniform}, bounds on the $p$-torsion exponent of the homotopy groups of spheres at odd primes.


Stably our knowledge is much less decisive.
From the unstable bound we learn that at odd primes
\[ \tors_p\left(\pi_n\Ss \right) \leq \frac{1}{2}n + O(1). \]
The first natively stable torsion bound was proved by Adams in \cite[Ch.6 Remark 1]{Ad:SHT}.
Therein he considered his eponymous spectral sequence, which has signature
\[ \H^{s,t}(\A) \cong \AE_2^{s,t} \Longrightarrow \pi_{t-s}\Ss_p \]
and provides a filtration on $\pi_n\Ss_p$ whose associated graded consists of sub-quotients of $\H^{s,t}(\A)$. Since the cohomology of the Steenrod algebra is an $\F_p$ vector space, Adams concluded that the torsion exponent of $\pi_n\Ss_p$ is at most the number of non-zero terms on the $E_\infty$-page contributing to that degree. As a corollary of the vanishing line on the $E_2$-page from \cite{AdamsPer} he obtained the bound 
\[ \tors_2\left( \pi_n \Ss \right) \leq \frac{1}{2}n + O(1). \]
The analogous vanishing line at odd primes is due to Luilevicius \cite{liul} and provides the bound
\[ \tors_p\left( \pi_n \Ss \right) \leq \frac{1}{2p-2}n + O(1). \]


Although the stable torsion bound has been improved several times since then,
each of these improvements proceeds in the same way:
by further confining the number of non-zero terms on
the Adams $E_\infty$-page contributing to a single degree. 
In \cite{DM3}, Davis and Mahowald used their technique of $\mathrm{bo}$-resolutions to give a measure of control over the $v_1$-local region in the Adams spectral sequence. 
As a consequence they obtained the bound
\[ \tors_2\left( \pi_n \Ss \right) \leq \frac{3}{10} n + O(\log(n)). \]
The corresponding bound at odd primes is due to Gonzalez
\[ \tors_p\left( \pi_n \Ss \right) \leq \frac{2p-1}{(2p-2)(p^2 - p - 1)} n + O(\log(n)) \]
appearing in \cite{Gonz00}.
The next improvement to the linear term appeared in \cite[Appendix B]{boundaries} where the author sharpened the analysis of the $E_\infty$ vanishing curve of the $\BP\langle 1 \rangle$-based Adams spectral sequence in order to obtain the bound
\[ \tors_p\left( \pi_n \Ss \right) \leq \frac{2p-1}{(2p-2)(2p^2 - 2)}n + o(n). \]
In this paper we offer the following theorem, which completely eliminates the linear term in the stable torsion bound.

\begin{thm} \label{main-thm}
  The $p$-torsion exponent of $\pi_n\Ss$ grows sublinearly in $n$.
\end{thm}

Our method is a modification of Adams' method where we instead analyze the Adams--Novikov spectral sequence.
The benefit of doing this is that (an alternative form of) the nilpotence theorem of \cite{DHS} tells us that the number of terms which contribute to a single stem grows sublinearly in $n$.
However, unlike the Adams $E_2$-page the $E_2$-page of the Adams--Novikov spectral sequence is not an $\F_p$ vector space. The improvement which makes \Cref{main-thm} possible is a novel torsion bound on the Adams--Novikov $E_2$-page which we prove in \Cref{prop:tor-bound}.

\begin{rmk}
  The proof strategy used in \Cref{main-thm} would only produce a sharp bound if there existed degrees where almost every possible hidden extension by $p$ in the Adams--Novikov spectral sequence occurred.
  Although this seems improbable, the author knows no technique which could rule this out.

  Conversely, the current lower bound of $\log_p(n)$ comes from the image of $J$.
  In view of \Cref{prop:tor-bound}, in order to beat this lower bound one would need a construction
  which systematically produces large numbers of hidden extensions in the Adams--Novikov spectral sequence.  
  \tqed
\end{rmk}

\subsection{Rank bounds}\ 

In contrast, much less has been said about the ranks of the stable stems
and no bound which makes use of the special features of the stable setting has previously appeared.
As with torsion bounds, the $E_2$-page of the Adams spectral sequence provides a simple, though in this case inexplicit, upper bound\footnote{Note that we consider the Adams $E_2$-page as a module over its polynomial subalgebra $\F_p[q_0]$ where $q_0$ is the class detecting $p$.}
\[ \mathrm{rank}_p(\pi_n\Ss) \leq \mathrm{rank}_{\F_p[q_0]}\left( \oplus_{t-s=n} \H^{s,t}(\A) \right). \]




In order to make this bound explicit we must estimate the growth of the cohomology of the Steenrod algebra.
Although this cohomology is incredibly complicated a simple upper bound is provided by the $E_1$-page of the May spectral sequence.
This being a polynomial algebra,
it is relatively relatively straightforward to estimate its growth,
proving the following theorem.


\begin{thm} \label{thm:stems-bound}
  Through the May $E_1$-page we have the following bound on the rank of the stable stems
  \[ \log\left( \mathrm{rank}_p(\pi_n\Ss) \right)\leq \log \begin{pmatrix} \text{rank of Adams} \\  E_2 \text{-page} \end{pmatrix} \leq \log \begin{pmatrix} \text{rank of May} \\ E_1 \text{-page} \end{pmatrix} =  \Theta(\log(n)^3) . \]
  In particular the rank of the stable stems grows subexponentially.
\end{thm}

In \Cref{sec:steenrod} we will examine the cohomology of the Steenrod algebra more closely,
proving that the bound provided by the May $E_1$-page is close to sharp.

\begin{thm} \label{thm:Adams-E2-bound}
  The rank of the cohomology of the Steenrod algebras grow as
  \[ \log \left( \mathrm{rank}_{\F_p[q_0]}\left( \bigoplus_{t-s \leq n} \H^{s,t}(\A) \right) \right) = \Theta(\log(n)^3). \]
\end{thm}

Together these theorems imply that the totality of the differentials in the May spectral sequence only manages to modify the implicit constant on the $\log(n)^3$ term and not the fundamental growth rate.
This state of affairs might be summarized by the following counter-intuitive slogan:
\begin{quote} \centering \emph{The May spectral sequence has few differentials.} \end{quote}

Building on this we examine the Adams--Novikov spectral sequence, which has signature
\[ \H^{s}(\Mfg;\,\omega^{\otimes u}) \cong \ANE_2^{s,2u} \Longrightarrow \pi_{2u-s}\Ss_{(p)}, \]
in \Cref{sec:mfg}.
Again we arrive at the same bound.

\begin{thm} \label{thm:AN-E2-bound}
  The rank of the $E_2$-page of the Adams--Novikov spectral sequence grows as
  \[ \log \left( \mathrm{rank}_p\left( \bigoplus_{2u-s \leq n} \H^{s}(\Mfg;\,\omega^{\otimes u}) \right) \right) = \Theta(\log(n)^3). \]
\end{thm}

In order to evaluate whether the upper bounds we have given are \emph{good} we need to discuss how fast we expect the rank of the stable stems to grow.
Based on an analysis of the likely fate of the telescope conjecture
we were led to the following, rather jarring, conjecture.

\begin{cnj}
  The Adams spectral sequence has few differentials in the sense that 
  \[ \log \left( \rank_p( \pi_n\Ss ) \right) = \Theta(\log(n)^3). \]
\end{cnj}

In \Cref{sec:motivic} we verify the analog of this conjecture for the $p$-complete motivic stable stems over $\C$.

\begin{thm} \label{thm:mot-size-intro}
  The growth of the $p$-complete motivic stable stems over $\C$ is described by 
  \[ \log \left( \rank_{\Z_p[\tau]}\left( \oplus_{s \leq n} \pi_{s,*}^{\C}(\Ss_p) \right) \right) = \Theta(\log(n)^3). \]
\end{thm}

The decisive factor in the motivic setting is the deformation parameter $\tau$ which closely links the cellular $p$-complete motivic category over $\C$ to the moduli of formal groups.\footnote{See \cite{ctau} for a precise incarnation of this link.}
Through this link we are able to lean heavily on \Cref{thm:AN-E2-bound} in proving \Cref{thm:mot-size-intro}.



\subsection{Known lower bounds}\ 

Adams' work on the image of $J$ provided the first infinite family of elements in the homotopy groups of spheres \cite{AdamsJIV}. This family consists\footnote{At $p=2$ things are more complicated, but not in a way that affects our conclusions.} of a copy of $\Z/p^{k+1}$ in each degree $n$ which is congruent to $-1$ modulo $2p-2$ where $k = v_p(n+1)$.
Altogether this family contributes $O(1)$ to the average rank of the $p$-local stable stems.
In fact, this lower bound remains essentially the state-of-the-art. 
In light of this we issue the following challenge:

\begin{challenge}
  Construct, or otherwise prove the existence of, a family of classes in the $p$-local stable stems with average size greater than $O(1)$. 
\end{challenge}

If instead of average rank we look at maximum rank,
then Oka's work on the divided beta family provides the first example of super-constant rank in the stable stems.

\begin{exm}
  \label{exm:oka}
  As a output of his study of ring structures on generalized Moore spectra in \cite{Okabeta},
  Oka constructed a collection of linearly independent classes $\beta_{tp^n/s} \in \pi_*\Ss$
  for $p \geq 5$, $t \geq 2$, $n \geq 3$ and $1 \leq s \leq 2^{n-2}p$.  
  Although this family only averages out to a constant number of generators in each degree,
  it is distributed less smoothly than the image of $J$.
  From this family we obtain the lower bound
  \[ \log\log(n) \lesssim \max_{j \leq n}\left( \mathrm{rank}_p( \pi_j \Ss) \right) . \]
  \tqed
\end{exm}

The author was quite surprised to find that not only are these the best lower bounds in the literature,
but even after substantial effort we were unable to improve the situation.
By this we do not mean to suggest that the stable stems are likely to be small, but instead to highlight how little is known.










\subsection{Conjectural lower bounds}\ 

It has long been recognized that if the telescope conjecture fails
it must fail violently.
In \Cref{sec:telescope},
drawing upon a web of conjectures by Mahowald, Ravenel and Schick
detailing the expected fate of the telescope conjecture,
we will bring this intuition to bear on the size of the stable stems.
What we find is that if the telescope conjecture fails in the expected way,
then in doing so it produces a sufficient supply of classes to saturate the upper bound provided by the Adams $E_2$-page.

\begin{thm} \label{thm:telescope-bound}
  If we assume that the telescope conjecture fails in the precise sense detailed by Conjectures \ref{conj:mrs-phase-1} and \ref{conj:mrs-collapse}, then
  \[ \log \left( \rank_p\left( \pi_n\Ss \right) \right) = \Theta ( \log(n)^3 ) . \]  
\end{thm}

In \Cref{sec:telescope} we will also explain how this relationship is robust in the sense that partial progress towards proving these conjectures will produce large numbers of elements in the stable stems.
As a consequence it becomes clear exactly what is at stake when we ask,
``How are big the stable stems?''
The size of the stable stems is a reflection of the fate of the telescope conjecture.

\subsection{Beyond torsion and rank}\ 

Torsion exponents and rank are only the beginning.
If we encode $p$-local finite abelian groups as tableaux where each column of height $k$ corresponds to a $\Z/p^k$ summand, then we have only discussed two measures of the shape of a tableau: height and width.
There are many other natural questions we might ask.
How does the number of blocks grow as a function of $n$?
What is the ratio of the longest row to the second longest row?
After an appropriate normalization do these tableaux converge to a limiting shape?
What shape?
All these might be viewed as shadows of the more fundamental,

\begin{qst} \label{qst:process}
  Is there a simple random process which produces tableaux with the same statistics and asymptotics as those obtained from the stable homotopy groups of spheres?
\end{qst}

Although a positive answer to this question is certainly out of reach, a random process which conjecturally meets the requirements of \Cref{qst:process} would already provide a great wealth of conjectures and expectations about the behavior of the stable stems---conjectures and expectations which are so sorely lacking.

\subsection{Integral bounds}
\label{subsec:integral}\ 

Up to this point we have been content to $p$-localize and consider torsion exponents and rank separately. This is somewhat unsatisfactory as it leaves the possibility that there may be emergent behavior in the integral case not visible at any individual prime.
This is not the case.

\begin{thm}
  The size of the stable stems satisfies the bound
  \[ \log\log |\pi_n\Ss| \lesssim \log(n)^3. \]
\end{thm}

\begin{proof}
  In order to bound the size of integral stems we need to bound the $p$-local size for all $p$ simultaneously.
  Using that $\alpha_1 \in \pi_{2p-3}\Ss$ is the first $p$-torsion element we only need to look at a finite number of primes at a time.
  Using the easy torsion bound (that $p^n$ acts by zero on $\pi_n\Ss_{(p)}$ for $n > 0$)
  we can reduce to understanding the $p$-ranks.
  Finally, the bound on the $p$-ranks we proved was based on analyzing the size of the May $E_1$-page
  and an examination of the degrees the polynomial generators on this $E_1$-page
  lets us conclude that this bound becomes stricter as $p$ becomes larger
  (i.e. the largest potential rank is at $p=2$).
  Altogether, we obtain 
\begin{align*}
  \log |\pi_n\Ss|
  &= \sum_{p} \log |\pi_n\Ss_{(p)}|
    = \sum_{p \leq n} \log |\pi_n\Ss_{(p)}|
    \leq \sum_{p \leq n} \log(p) \cdot n \cdot \mathrm{rank}_p( \pi_n\Ss )  \\
  &\leq \log(n) n^2 \cdot \exp( \Theta(\log(n)^3) )
  =  \exp( \Theta(\log(n)^3) ). \qedhere
\end{align*}
\end{proof}

\subsection{Unstable bounds}\ 

In the unstable context, the growth of homotopy groups has received more attention.
The sharpest bounds on the ranks of the unstable homotopy groups of a space known before the present work
are due to Boyde \cite{Boyde}, who improved on an earlier bound of Henn \cite{Henn-bound}.
Notably, the bounds of Boyde and Henn are exponential, and it has been an open question whether the ranks of unstable homotopy groups of spheres grow subexponentially \cite[Question 1.7]{Huang-Wu}.

In \Cref{sec:unstable}, which is joint with Andrew Senger,
we prove the following bound on the unstable homotopy groups of a simply-connected space of finite type.


\begin{thm}[Burklund--Senger]\label{thm:app-main}
  Let $X$ denote a simply-connected space of finite type. Then there is a bound
  \[\log \left( \rank_p ( \pi_n X ) \right) \leq  \log \left(\sum_{i=1} ^{n-1} \rank_{\F_p} (\H_{i} (\Omega X;\, \F_p)\right) + O(\log(n)^3).\]
\end{thm}

An immediate consequence of \Cref{thm:app-main} is that the ranks of the unstable homotopy groups of a sphere grow no faster than $\exp(O(\log(n))^3)$, hence subexponentially.
\Cref{thm:app-main} verifies a conjecture of Henn \cite[Conjecture on p. 237]{Henn-bound},
which asserts that the exponential term in the growth of $\rank_p (\pi_n (X))$ is \emph{equal} to that of $\rank_{\F_p} (\H_i (\Omega X; \F_p))$ for a simply-connected finite-type space $X$.\footnote{Note that \Cref{thm:app-main} only gives an inequality in one direction. The reverse inequality is due to Iriye \cite{Iriye}.}

The key novelty in the proof of \Cref{thm:app-main} is the use of an EHP-type reduction to the stable setting where we now have the subexponential bounds from the body of the paper available.
Previously, the flow of information had been in the other direction with stable bounds typically being proved as stabilizations of unstable bounds proved by other means.

In \Cref{sec:un-tor-bd}, we consider the state of knowledge on torsion bounds for unstable homotopy groups.
While the Cohen--Moore--Neisendorfer theorem provides optimal bounds on the $p$-torsion exponents of unstable homotopy groups of spheres for $p$ odd, much less is known for more general spaces.
The state-of-the-art is Barratt's theorems from \cite{Barratt},
which focus on spaces which are suspensions.
As a simple consequence of Goodwillie calculus, we prove the following bound
which has no such restriction.

\begin{thm}[Burklund--Senger]
  If $X$ is $s$-connected with $s \geq 1$
  and the identity map on $\Sigma^\infty X$ has order $p^m$, then
  \[ \tors_p\left( \pi_nX \right) \leq \frac{(m+1)}{s}n. \]
\end{thm}

%
%
%
%
%

\subsection{Conventions}
\label{subec:cnv}\ 

Outside of the brief Subsection \ref{subsec:integral}
we work locally at fixed prime $p$.
In the introduction we have avoided giving precise constants for brevity.
Going forward, our statements will be more precise and three constants recur often enough
that we define them here:

\begin{align}
  \label{eqn:constants} K_1 &= \frac{2}{75 \log(p)^2} & K_2 &= \frac{9 + 4\sqrt{2}}{294 \log(p)^2} & K_3 &= \frac{1}{6\log(p)^2} \\[3mm] 
  \nonumber K_1 &\approx \frac{0.0267}{\log(p)^2} & K_2 &\approx \frac{0.0499}{\log(p)^2} & K_3 &\approx \frac{0.1667}{\log(p)^2}.
\end{align}

\subsection{Acknowledgments}\ 

This paper contains a number of results first announced in the summer of 2020.
We apologize for the long delay and to all the people who were told this would appear ``soon'' along the way.
The author would like to thank
Ishan Levy,
Mike Hopkins,
Haynes Miller,
Piotr Pstr\k{a}gowski,
Doug Ravenel,
Tomer Schlank,
Jonathan Tidor and
Roger Van Peski
for helpful conversations related to this work.

During the course of this work, Andrew Senger was supported by NSF Grant DMS-2103236.

%% file: torexp.tex
In this section we prove that the $p$-torsion exponent of the stable stems grows sublinearly in $n$.
In fact, what we prove here is the more precise \Cref{thm:sublinear-tor-precise} which identifies the source of this sublinear upper bound with the $E_\infty$ vanishing curve of the Adams--Novikov sseq. Before proceeding we briefly review what is known and expected of this curve.

\begin{dfn}
  Let $\gbp(n)$ denote the $E_\infty$ vanishing curve of the Adams--Novikov sseq
  \[ \gbp(n) \coloneqq \max \{ s \ |\ \ANE_\infty^{s, k+s} \neq 0 \text{ for some } k \leq n \}. \]
  As we will not consider vanishing curves associated to other ring spectra
  we will sometimes abuse notation, writing $g(n)$ for $\gbp(n)$.
  \tqed
\end{dfn}

A strong form of the nilpotence theorem of \cite{DHS} first worked out by Hopkins and Smith shows that the function $\gbp(n)$ grows sublinearly\footnote{A proof of this result has gone on to appear in \cite{Akhil}}.
Based on heuristic arguments related to the telescope conjecture they made the following conjecture on the leading order asymptotics of this function.

\begin{cnj}[Hopkins--Smith vanishing conjecture] \label{cnj:HS-vc}
  \[ \gbp(n) = n^{\frac{1}{2} + o(1)}. \]
\end{cnj}

\begin{rmk}
  At this point we pause to point out that
  this conjecture would imply a strong, quantitative form of the Nishida nilpotence theorem:
  for every $\alpha$ in degree $n>0$ 
  \[ \alpha^{n^{1 + o(1)}} = 0. \]
  Conversely, one approach to proving \Cref{cnj:HS-vc} is through quantitative nilpotence theorems.
  \tqed
\end{rmk}

It is in terms of the sublinear function $\gbp(n)$ that we bound the $p$-torsion exponent of the stable stems.

\begin{thm} \label{thm:sublinear-tor-precise}
  The $p$-torsion exponent of the stable stems is bounded by the following linear combination of $\gbp(n)$ and a logarithmic error term
  \begin{align*}
    \tors_2\left( \pi_n\Ss \right) &\leq \frac{5}{4} \gbp(n) + \log_2(n) + 2 & &\text{ and } \\
    \tors_p\left( \pi_n\Ss \right) &\leq \frac{p}{2(p-1)^2} \gbp(n) + \log_p(n) + 1 & &\text{ for } p \text{ odd}. 
  \end{align*}                                                                                          
  In particular, since $\gbp(n)$ grows sublinearly
  the $p$-torsion exponent of the stable stems grows sublinearly as well.
\end{thm}

As discussed in the introduction, we prove this theorem by bounding the torsion exponent of the Adams--Novikov $E_2$-page and then using $\gbp(n)$ to bound the number of terms on the $E_\infty$-page contributing to each stable stem.

\begin{prop} \label{prop:tor-bound}
  $ (\ell^u - 1) \cdot \ANE_2^{s,2u} = 0 $ for any $\ell \in \Z_{(p)}^\times$.
\end{prop}

In order to make \Cref{prop:tor-bound} more usable we restate it in the following explicit form.

\begin{cor} \label{rmk:tor-bound-simple}
  \begin{align*}
    \tors_2\left( \ANE_2^{s,2u} \right)
    &\leq
      \begin{cases}
        1 & u \text{ odd},\\
        2 + |u|_2 & u \text{ even}
      \end{cases} & &\text{ and } \\
    \tors_p\left( \ANE_2^{s,2u} \right)
    &\leq
      \begin{cases}
        0 & u \not\equiv 0 \pmod{p-1}, \\        
        1 + |u|_p & u \equiv 0 \pmod{p-1}
      \end{cases} & &\text{ for } p \text{ odd}.
  \end{align*}  
\end{cor}

\begin{rmk}
  The classes on the $1$-line of the Adams--Novikov $E_2$-term saturate the bound from \Cref{rmk:tor-bound-simple} implying that it is sharp.
  \tqed
\end{rmk}

\begin{rec}
  Quillen's work on complex cobordism allows us describe the Adams--Novikov $E_2$-page as the cohomology of the moduli of formal groups.\footnote{Note that as we are working $p$-locally throughout this paper $\Mfg$ lives over $\Spec(\Z_{(p)})$ and classifies formal groups over $p$-local rings.} Specifically, we have a line bundle $\omega$ on the moduli of formal groups, $\Mfg$, given by sending a formal group $\G$ to the inverse of its Lie algebra and an isomorphism
  \[ \ANE_2^{s,2u} \cong \mathrm{H}^{s}( \Mfg;\, \omega^{\otimes u} ). \]
  \tqed
\end{rec}

For our purposes it will be convenient to reformulate this isomorphism slightly.
If we pass to the stack $\Mfg^{\omega=1}$ of formal groups equipped with a trivialization of their Lie algebra which sits as a $\G_m$-torsor over $\Mfg$, then we have isomorphisms
\[ \ANE_2^{s,*} \cong \mathrm{H}^{s}( \Mfg^{\omega=1} ) \cong \oplus_u \mathrm{H}^{s}( \Mfg;\, \omega^{\otimes u} ) \]
where the $u$-grading on the cohomology is recorded by the $\mathbb{G}_m$-action coming from rescaling the trivialization. In order words, $a \in \G_m(X)$ acts on an element $x \in \mathrm{H}^{s}( (\Mfg)_X;\, \omega^{\otimes u} )$ by sending it to $a^{u}x$. 
Specializing this formula to $\G_m(\Z_{(p)})$
we see that in order to prove \Cref{prop:tor-bound} it will suffice to
give a trivialization of the action of $\mathbb{G}_m(\Z_{(p)})$ on $\Mfg^{\omega=1}$.

\begin{prop}
  The restriction of the $\mathbb{G}_m$ action on $\Mfg^{\omega=1}$ to the discrete group
  $\mathbb{G}_m(\Z_{(p)})$ is trivializable.
\end{prop}

We thank Piotr Pstragowski for suggesting that the following proof can be simplified by working geometrically.

\begin{proof}
  Since we are working with $p$-local rings
  every formal group is automatically a formal $\Z_{(p)}$-module.
  This provides us with natural automorphisms $[\ell] : \G \to \G$ 
  for each $\ell \in \Z_{(p)}^\times$ and formal group $\G$.
  Taken over all choices of $\ell \in \Z_{(p)}^\times$, these natural automorphisms assemble into
  an action of $\mathrm{B}\Z_{(p)}^\times$ on $\Mfg$.
  Multiplication by $\ell$ on $\G$ acts as multiplication by $\ell$ on the Lie algebra of $\G$,
  therefore we can upgrade the map
  \[ \Mfg \to B\G_m \]
  sending a formal group to its Lie algebra to a $\mathrm{B}\Z_{(p)}^\times$-equivariant map
  where $\ell \in \Z_{(p)}^\times$ acts on a line bundle $\mathcal{L}$ as multiplication by $\ell$.
  Taking quotients by these actions we obtain a pullback square
  \begin{center}
    \begin{tikzcd}
      \Mfg \ar[r] \ar[d] \pullback & \Mfg/\mathrm{B}\Z_{(p)}^\times \ar[d] \\
      B\G_{m} \ar[r] & B\left(\G_{m}/\Z_{(p)}^\times \right)
    \end{tikzcd}
  \end{center}
  which witnesses that the $\G_m$ action on $\Mfg^{\omega=1}$ is restricted from a $\G_{m}/\Z_{(p)}^\times$ action.
  Consequently, $\G_m(\Z_{(p)})$ acts trivially on $\Mfg^{\omega=1}$ as desired.
  \qedhere

\end{proof}

With the bound on the $p$-torsion exponent in the Adams--Novikov $E_2$-page proved we turn to proving \Cref{thm:sublinear-tor-precise}.

\begin{proof}[Proof (of \Cref{thm:sublinear-tor-precise}).]  
  Since $\ANE_\infty^{s,2u}$ is a subquotient of $\ANE_2^{s,2u}$
  we can obtain a bound on the $p$-torsion exponent of $\pi_n\Ss$
  by amalgamating the bounds from \Cref{prop:tor-bound} as $s$ ranges between $1$\footnote{We start with 1 because the Adams--Novikov 0-line for the sphere is empty outside $t=0$.} and $g(n)$.
  Using the explicit form of these bounds from \Cref{rmk:tor-bound-simple} we obtain
  \begin{align*}
    \tors_2\left( \pi_n\Ss \right)
    &\leq \sum_{s=1}^{g(n)} \tors_2\left( \ANE_2^{s,n+s} \right)
      \leq \sum_{i = \left\lfloor \frac{n}{2} \right\rfloor + 1}^{\frac{n+g(n)}{2}}
      \left( 1 + \left| i \right|_2 + \begin{cases} 1 & 2|i \\ 0 & 2\!\!\!\not| i \end{cases} \right)
                                                      & & \text{and} \\
    \tors_p\left( \pi_n\Ss \right)
    &\leq \sum_{s=1}^{g(n)} \tors_p\left( \ANE_2^{s,n+s} \right)
      \leq \sum_{i = \left\lfloor \frac{n}{2p-2} \right\rfloor + 1}^{\frac{n+g(n)}{2p-2}} \left( 1 + \left| i \right|_p \right) & & \text{for } p \text{ odd}. \\
  \end{align*}                                                                                                                                   
  In order to complete the proof we need to simplify the final term.
  We do this separately at $p=2$ and at odd primes.
  Starting with $p=2$, we use \Cref{lem:some-counting} to obtain
  \begin{align*}
    \sum_{i = \left\lfloor \frac{n}{2} \right\rfloor + 1}^{\frac{n+g(n)}{2}}
    \left( 1 + \left| i \right|_2 +
    \begin{cases}
      1 & 2|i \\
      0 & 2\!\!\!\not| i
    \end{cases} \right)
        &\leq \left( 2\left( \frac{n+g(n)}{2} - \left\lfloor \frac{n}{2} \right\rfloor \right)
          + \log_2\left( \frac{n+g(n)}{2} \right) \right) \\
        &\quad+ \left| \left\{ \text{even } i \in \left[ \left\lfloor \frac{n}{2} \right\rfloor + 1,\ \frac{n+g(n)}{2} \right] \right\} \right|. 
  \end{align*}
  Using the fact that $g(n) \leq n$ (which is already visible on the $E_2$-page)
  we can further simplify
  \begin{align*}
    \tors_2\left( \pi_n\Ss \right)
    &\leq  2\left( \frac{n+g(n)}{2} - \left\lfloor \frac{n}{2} \right\rfloor \right)
      + \log_2\left( \frac{n+g(n)}{2} \right) 
      + \left\lceil \frac{1}{2} \left( \frac{n+g(n)}{2} - \left\lfloor \frac{n}{2} \right\rfloor \right) \right\rceil \\
    &\leq  g(n) + 1 
      + \log_2\left( n \right) 
      + \left( \frac{g(n)}{4}  + \frac{n}{4} - \frac{1}{2} \left\lfloor \frac{n}{2} \right\rfloor \right) + \frac{3}{4} \\
    &\leq  \frac{5}{4} g(n) + \log_2\left( n \right) + 2.
  \end{align*}
  At odd primes we follow the same procedure: using \Cref{lem:some-counting} then simplifying.
  \begin{align*}
    \sum_{i = \left\lfloor \frac{n}{2p-2} \right\rfloor + 1}^{\frac{n+g(n)}{2p-2}}
    \left( 1 + \left| i \right|_p \right)
    &\leq \frac{p}{p-1} \left( \frac{n+g(n)}{2p-2} - \left\lfloor \frac{n}{2p-2} \right\rfloor \right)
      + \log_p\left(\frac{n+g(n)}{2p-2}  \right) \\
    &\leq \frac{p}{p-1} \left( \frac{g(n)}{2p-2} + 1 \right) + \log_p\left(\frac{n}{p-1} \right) \\
    &\leq \frac{p}{2(p-1)^2} g(n) + \log_p(n) + 1. 
  \end{align*}
\end{proof}

\begin{lem} \label{lem:some-counting}
  For $0 \leq a < b$ we have
  \[ \sum_{i=a+1}^{b} (1 + |i|_p) \leq \frac{p}{p-1}(b-a) + \log_p(b). \]
\end{lem}

\begin{proof}
  Let $K$ be the maximum over $a+1 \leq i \leq b$ of $|i|_p$.
  \begin{align*}
  \sum_{i=a+1}^{b} (1 + |i|_p)
    &= (b-a) + \sum_{k = 1}^K \left| \{a+1 \leq i \leq b \text{ such that } p^k \text{ divides } i \} \right| \\
    &\leq (b-a) + \sum_{k = 1}^K \left( 1 + \frac{b-a}{p^k} \right) 
    \leq (b-a) + K + \frac{1}{p-1}(b-a) \\
    &\leq \frac{p}{p-1}(b-a) + \log_p(b).
  \end{align*}
\end{proof}


%% file: val.tex
In the introduction we confined our discussion to the sphere and its homotopy groups.
Going forward, we will need to broaden our scope
as our proofs pass through a number of auxiliary objects.
In order to organize our arguments we introduce the notion of a valuation on a stable category.

\begin{dfn} \label{dfn:height}
  A \emph{valuation} on a stable category $\CC$ is a function
  \[ | - | : \CC^\simeq \longrightarrow \Gamma \]  
  with values in an partially ordered
  $\Z[\sigma^{\pm1}]$-module $\Gamma$
  satisfying the following conditions:
  \begin{enumerate}
  \item For each $X$ we have $|X| \geq 0$.
  \item For each $X$ and $Y$ we have $|X \oplus Y| = |X| + |Y|$.
  \item For each cofiber sequence $X \to Y \to Z$ we have $|Y| \leq |X| + |Z|$.
  \item For each $X$ we have $\sigma \cdot |X| = |\Sigma X|$.
  \end{enumerate}
  We say that the valuation $|-|$ is \emph{rotation invariant} if $\sigma$ acts as $1$.  
  \tqed
\end{dfn}

The basic example of a valuation is the Hilbert-Poincare series of a module over a field.

\begin{exm} \label{exm:HP-series}
  Given a field $k$ there is a valuation $|-|$ on $\Mod(k)$ given by sending $V$ to the Laurent series
  \[ \sum_n \rank_k\left(  \pi_nV \right) \cdot t^n \]
  The action of $\sigma$ on Laurent series is given by multiplication by $t$
  and the partial order is by pointwise inequality.
  \tqed
\end{exm}

\begin{exm} \label{exm:val-mapsout}
  Given a stable category $\CC$
  and an object $X \in \CC$ 
  such that $[X,Y]$ is a finitely generated $\Z_{(p)}$-module for every $Y \in \CC$
  we define a valuation
  \begin{align*}
    |-|_X : \CC^{\simeq} \quad &\longrightarrow \quad \Map(\Z, \Z) \\
    Y \quad  &\mapsto \quad \left( n \mapsto \rank_p [\Sigma^n X, Y] \right) 
  \end{align*}
  where the partial order is by pointwise inequality and
  $\sigma$ sends a function $f(n)$ to $f(n-1)$.
  \tqed
\end{exm}

\begin{lem} \label{lem:pullback-val}
  If $F : \CC \to \DD$ is an exact functor of stable categories and $|-|$ is a valuation on $\DD$,
  then $|F(-)|$ is a valuation on $\CC$.
\end{lem}

\begin{proof}
  Since $F$ is assumed exact it preserves sums, cofiber sequences and suspensions which suffices to check the conditions on a valuation.
\end{proof}

The valuations we consider in this paper are for the most part constructed in a straightforward way using a combination of Examples \ref{exm:HP-series}, \ref{exm:val-mapsout} and \Cref{lem:pullback-val}.

\begin{lem}
  \label{lem:ur-AH}
  Suppose that $\CC$ is a stable category with a valuation $| - |$.
  If $X \in \CC$ is in the thick subcategory generated by $Y$,
  then there exists some $P(\sigma) \in \N[\sigma^{\pm1}]$
  depending on $X$ and $Y$ such that
  \[ |X| \leq P(\sigma) \cdot |Y|. \]
\end{lem}

\begin{proof}
  The given condition on $X$ is equivalent to saying that $X$ can constructed in finitely many steps beginning with $Y$ and applying the operations
  \begin{itemize}
  \item suspend an integer number of times,
  \item take an extension,
  \item split an idempotent.
  \end{itemize}
  The conditions of \Cref{dfn:height} allows us to track how valuation changes under these operations, proving the lemma.
  In fact, $P(\sigma)$ can be taken to record the copies of $\Sigma^k Y$ used in constructing $X$.
\end{proof}

\begin{rmk}
  \Cref{lem:ur-AH} can be viewed as recording a weak Morita invariance property of valuations.
  For example, if $|-|$ is rotation invariant
  and $X$ and $Y$ generate the same thick subcategory of $\CC$,
  then their valuations differ by at most a scalar factor.
  \tqed
\end{rmk}

\subsection{Valuations on spectra}\

Using the language of valuations we can reformulate \Cref{thm:stems-bound}
as a bound on a valuation measuring the rank of the stable stems.

\begin{cnstr} \label{dfn:gen-funcs}
  We can construct two natural valuations
  on the category of compact spectra
  each valued in functions $\Z \to \Z$
  with pointwise inequalities
  \begin{align*}
    \varpi(X,\, n) &\coloneqq \rank_p(\pi_nX) \\
    h(X,\, n) &\coloneqq \rank_p(\H_n(X;\,\F_p)).
  \end{align*}
  For brevity we write $\varpi(n)$ for $\varpi(\Ss,\,n)$.
  \tqed
\end{cnstr}

Working with the valuation $\varpi$ can be somewhat inconvenient as it might vary wildly in $n$.
For example we do not even know whether $\varpi(n)$ is non-zero for all $n \gg 0$.
In order to pave over this issue we will introduce another pair of valuations
$\varpi^c$ and $h^c$ which are cumulative.

\begin{cnstr}
  Since $\varpi$ and $h$ takes values in
  functions which are zero for $n \ll 0$
  we can define valuations
  \[ \varpi^c(X,\, n) \coloneqq \sum_{k \leq n} \varpi(X,\, k) \qquad \text{ and } \qquad h^c(X,\, n) \coloneqq \sum_{k \leq n } h(X,\, k) \]
  which take values in the subset of non-decreasing functions $\Z \to \Z$.
  \tqed
\end{cnstr}

When working with $\varpi^c$ we often suppress rotation using the inequality
\[ \sigma \cdot \varpi^c(X,\,n) \leq \varpi^c(X,\,n). \]


\begin{rmk}
  As we continue it will be important for us to be able to work with certain, well-behaved, infinite objects.
  The natural class of objects to which the definitions of $\varpi$ and $h$ extend are those spectra which are \emph{almost compact} in the sense of \cite[Definition C.6.4.1]{SAG}.
  This condition is equivalent to asking that a spectrum $X$ be bounded below and have only finitely many cells in each degree.
  \tqed
\end{rmk}

Using the existence of minimal cell structures on $p$-local almost compact spectra we obtain the following corollary to \Cref{lem:ur-AH}.

\begin{cor} \label{cor:AH}
  Given an almost compact spectrum $X$
  \[ \varpi(X,\,n) \leq h(X,\,n) * \varpi(n) \]
  where $*$ denotes convolution.
  If $X$ is in addition connective, then 
  \[ \varpi^c(X,\,n) \leq h^c(X,\,n) \cdot \varpi^c(n) \]
  where $\cdot$ denotes pointwise multiplication.
\end{cor}

We highlight that the real power of this corollary isn't in providing an upper bound on the homotopy of a spectrum $X$, but its the ability to provide lower bounds on the homotopy of the sphere using an auxiliary spectrum.

\begin{rmk}
  The two valuations $\varpi$ and $h$ defined in this subsection highlight a dichotomy
  which emerges between
  valuations which measure stems (or some similarly complicated collection of groups)
  and valuations which measure the cells used to construct an object.
  \tqed
\end{rmk}



\subsection{Valuations on $\A$-comodules}\

As preparation for the proof of \Cref{thm:Adams-E2-bound},
which estimates the growth of the cohomology of the Steenrod algebra,
we introduce valuations $\varpi_{\A}$ and $h_{\A}$ on the category of $\A$-comodules,
$\check{\mathcal{D}}(\A)$,
which are analogous to $\varpi$ and $h$.

\begin{cnstr}  
  Given an $X \in \check{\mathcal{D}}(\A)^{\omega}$ let
  \[ \varpi_{\A}(X,\, n) \coloneqq \mathrm{rank}_{\F_p[q_0]}\left( \oplus_{t-s = n} \H^{s,t}(\A;\,X) \right). \]
  Analyzing the long exact sequence on cohomology groups we see that $\varpi_{\A}$ is a valuation with values in functions $\Z \to \Z$ and pointwise inequalities.
  
  As $\varpi_\A$ takes values in
  functions which are zero for $n \ll 0$
  we have a cumulative analog $\varpi^c_\A(X,\, n) \coloneqq \sum_{k\leq n} \varpi_\A(X,\, k)$.
  For brevity we write $\varpi_{\A}(n)$ for $\varpi_{\A}(\F_p, n)$.
  \tqed
\end{cnstr}

If $X$ is a compact spectrum,
then the cohomology of the $\A$-comodule $(\F_p)_*(X)$ is
the $E_2$-page of the Adams sseq for the homotopy groups of $X$
and $q_0$ is a class detecting $p$.
From this we can read off the inequality
\[ \varpi(X,\,n) \leq \varpi_\A((\F_p)_*(X),\, n) \]
which motivated us to introduce $\varpi_\A$.

The fact that $\H^{s,t}(\A)$ vanishes for $t-s <0$ lets us put a weight structure\footnote{See \cite[Section 1.4]{PalmieriBook} for a discussion and construction of this weight structure.} on $\check{\mathcal{D}}(\A)^\omega$ where the weight zero objects are sums of copies of $\Sigma^{-s,s}\F_p$.
On compact objects we can detect the cells of a comodule by tensoring with $\A$ and computing cohomology.\footnote{equivalently this is just looking at the underlying graded $\F_p$-module.}

\begin{cnstr}
  We can package the bidegrees of the cells of a comodule into a valuation 
  \[ h_{\A}(X,\, n) \coloneqq \mathrm{rank} \left( \oplus_{t-s = n} \H^{s,t}(\A;\, \A \otimes X) \right) \]
  which measure the number of cells in $t-s = n$.
  As usual we also have an associated cumulative valuation $h^c(X,\,n)$
  which measures number of cells through $t-s\leq n$.
  \tqed
\end{cnstr}

As in the case of spectra the natural class of objects to which the definitions of $\varpi_{\A}$ and $h_{\A}$ extend are those comodules which are almost compact in the sense that they have only finitely many cells with $t-s \leq n$ for any value of $n$.

Using the existence of minimal cell structures on almost compact objects in $\check{\mathcal{D}}(\A)$
we obtain another corollary of \Cref{lem:ur-AH} analogous to \Cref{cor:AH}.

\begin{cor} \label{ah-bound}
  Given an almost compact $\A$-comodule $X$
  \[ \varpi_\A(X,\,n) \leq h_\A(X,\,n) * \varpi_\A(n). \]
  If in addition $X$ is concentrated in non-negative degrees, then
  \[  \varpi_{\A}^c(X,\,n) \leq h_{\A}^c(X,\,n) \cdot \varpi_{\A}^c(n). \]
\end{cor}



%% file: steenrod.tex
In this section we analyze the asymptotic growth of the cohomology of the Steenrod algebra in order to prove Theorems \ref{thm:stems-bound} and \ref{thm:Adams-E2-bound}.


\begin{thm} \label{thm:steenrod-precise}
  The rank of the cohomology of the Steenrod algebra grow quasi-polynomially in the topological degree.
  In particular, it satisfies the following pair of bounds:
  \begin{align}
    \log \left( \varpi^c_\A(n) \right) &\leq K_3 \log(n)^3 + O(\log\log(n)\log(n)^2) \\
    \log \left( \varpi^c_\A(n) \right) &\geq K_2 \log(n)^3 + O(\log\log(n)\log(n)^2).
  \end{align}
  where $K_3$ and $K_2$ are the constants given in \cref{eqn:constants}.
\end{thm}

The upper bound is easier to prove and comes directly from the $E_1$-page of the May sseq.

\begin{proof}[Proof (of \Cref{thm:steenrod-precise}(1)).]
  We split into cases based on the parity of $p$.
  If $p$ is even, then the the May $E_1$-page is the polynomial algebra
  \[ \F_2[h_{i,j}\ |\ i \geq 1, j \geq 0] \]
  where $h_{i,j}$ is in a tri-degree which contribute to $n = 2^{i+j} - 2^j - 1$.
  Since $h_{1,0}$ detects $q_0$ we have the bound
  \[ \varpi^c_\A(n) = \rank_{\F_2[q_0]} \left( \oplus_{t-s \leq n} \H^{s,t}(\A) \right) \leq \rank_{\F_2[h_{1,0}]} \left( \oplus_{t-s \leq n} {}^{\mathrm{May}}E_1^{s,t,u} \right) \]
  We estimate the growth of the rank of polynomial algebras of this sort in \Cref{app:combo}
  where \Cref{cor:may-flexible} provides the desired upper bound.


  If $p$ is odd, then the May $E_1$-page takes the more complicated form
  \[ \F_p[q_i\ |\ i \geq 0][ h_{i,j}, b_{i,j}\ |\ i \geq 1, j \geq 0] / (h_{i,j}^2) \]
  where the topological degrees are
  $|q_i| = 2p^i - 2$,
  $|h_{i,j}| = 2p^{i+j} - 2p^{j} - 1$ and
  $|b_{i,j}| = 2p^{i+j+1} - 2p^{j+1} - 2$.
  Examining the contribution of each pair $h_{i,j}, b_{i,j}$ we can replace the May $E_1$-page with a polynomial algebra including the $q_i$ together with \emph{polynomial} generators $h_{i,j}$ and the cumulative rank will only increase. Again we can bound the growth of this algebra using \Cref{cor:may-flexible}.
\end{proof}

The remainder of this section is devoted to the, more involved, proof of \Cref{thm:steenrod-precise}(2).
This lower bound comes from the interaction of three relatively simple observations.
The first is that \Cref{ah-bound}
implies we only need to find some $\A$-comodule $X$ with large cohomology.
The second is that although there are very few comodules for which we have a complete description of their cohomology, there exists a surprisingly rich collection for which we understand their $q_h$-localized cohomology---thanks to \cite[Section 2]{MRS}\footnote{The class which we denote $q_h$ is denoted $v_h$ in loc. cit.}.
The final observation is that it is not particularly difficult to lift $q_h$-local cohomology classes to the non-local setting.






\begin{dfn}
  Let $y(h)$ denote the following subalgebra of $\A$
  \begin{align*}
    y(h) &\coloneqq \F_p[\zeta_1,\dots,\zeta_h] \subseteq \A & \quad &(\text{ for } p \text{ even } ) \\
    y(h) &\coloneqq \F_p[\xi_1,\dots,\xi_h]\langle \tau_0, \dots, \tau_{h-1} \rangle \subseteq \A & \quad &(\text{ for } p \text{ odd } )
  \end{align*}
  considered as an $\A$-comodule algebra.
  \tqed
\end{dfn}

Although the cohomology algebra of $y(h)$ is not explicitly known,
its localization at the class $q_h$ in $(1, 2p^h-1)$ has been computed by Mahowald, Ravenel and Schick.

\begin{prop}[{\cite[Equation 2.20]{MRS}}] \label{prop:mrs-e2}
  The localization of the cohomology of $y(h)$ at $q_h$ is given by
  \begin{align*}
    &\F_p[q_h^{\pm 1}, q_{h+1}, \dots, q_{2h}][ h_{i,j} \ |\ i > h,\, 0 \leq j \leq h-1 ] & &\text{ if } p \text{ is even} \\
    &\F_p[q_h^{\pm 1}, q_{h+1}, \dots, q_{2h}][ h_{i,j},\, b_{i,j} \ |\ i > h,\,  0 \leq j \leq h-1 ]/(h_{i,j}^2) & &\text{ if } p \text{ is odd}
    \end{align*}
  where $|q_k| = (1, 2p^k - 1)$,
  $|h_{i,j}| = (1,\ 2p^{i+j} - 2p^{j})$ and
  $|b_{i,j}| = (2, 2p^{1+i+j} - 2p^{j+1})$.
\end{prop}

\begin{rmk}
  The $\A$-comodule which we call $y(h)$ in this section is the $\F_p$-homology of a spectrum $y(h)$ and the $q_h$-local cohomology groups calculated in \cite[Equation 2.20]{MRS} are the the $E_2$-page of the localized Adams sseq which computes $\pi_* (y(h)[v_h^{-1}])$.
  We will return to this point in \Cref{sec:telescope}.
  \tqed
\end{rmk}

Since $y(h)$ is a comodule algebra with a large polynomial subalgebra in its $q_h$-local cohomology
we can produce a large polynomial algebra in $\H^{s,t}(\A;\, y(h))$ by lifting polynomial generators to the non-localized cohomology.





\begin{lem} \label{lem:yn-size}
  The cohomology of $y(h)$ contains a polynomial subalgebra on classes $x_{i,j}$ which live in topological degree $12p^{1+i+j} - 10p^{1+i-h+j} - 2p^{1+j} - 2$ where $i > h$ and $0 \leq j \leq h-1$.
\end{lem}

\begin{proof}
  Since \Cref{prop:mrs-e2} provides a large polynomial algebra in the localized cohomology groups, all we need to do is find lifts of these polynomial generators.
  By considering the May sseq for $y(h)/q_h$ we can conclude that the localization map induces an isomorphism for $ s-1 \geq (2p^{h+1} - 3)^{-1}(t-s) $.
  This means we can let the generator $x_{i,j}$ be the (unique) lift of $q_h^wb_{i,j}$ where $w = 5p^{1+i-h+j}$ since
  \begin{align*}    
    (5p^{1+i-h+j} + 2) - 1 \geq  (2p^{h+1} - 3)^{-1}\left( 5p^{1+i-h+j}(2p^h - 2) + (2p^{1+i+j} - 2p^{j+1} - 2) \right).
  \end{align*}   
  

\end{proof}

\begin{proof}[Proof (of \Cref{thm:steenrod-precise}(2)).]
  From \Cref{ah-bound} we have lower bounds
  \[ \varpi_{\A}^c(y(h),\,n) \leq h_{\A}^c(y(h),\, n) \cdot \varpi_{\A}^c(n). \]
  Since $y(h)$ is a subobject of $\A$ we also have
  \[ h_{\A}^c(y(h),\, n) \leq h_{\A}^c(\A,\, n) \]
  As $h$ varies we obtain the single inequality
  \[ \max_h \left\{ \varpi_{\A}^c(y(h),\,n) \right\}\leq h_{\A}^c(\A,\, n) \cdot \varpi_{\A}^c(n). \]
  To finish the proof we need to turn the large polynomial subalgebra from \Cref{lem:yn-size} into a lower bound on the cohomology of $y(h)$ and optimize the choice of $h$ as a function of $n$.
  All this is accomplished in \Cref{lem:yn-e2-basic} and \Cref{cor:yn-e2-flexible}.
  In summary we arrive at
  \[ K_2 \log(n)^3 + O(\log\log(n)\log(n)^2) \leq \log (h_{\A}^c(\A,\, n)) + \log (\varpi_{\A}^c(n)). \]
  From \Cref{exm:steenrod-size} we know that $\log (h_{\A}^c(\A,\, n)) = \Theta(\log(n)^2)$ which grows slowly enough that we can subtract it across and conclude.
\end{proof}

\begin{rmk}
  Throughout this section we have made various replacements,
  most notably dropping the polynomial generators $q_{h+i}$ between
  \Cref{prop:mrs-e2} and \Cref{lem:yn-size},
  which might at first suggest that the lower bound in \Cref{thm:steenrod-precise}
  could be improved by a more careful analysis.
  We remark that the flexibility provided by \Cref{cor:yn-e2-flexible}
  is sufficient to conclude that a more careful analysis would not improve the constant $K_2$.
    

  As an example:
  suppose we added in $q_i$ for all $i \geq 0$ not just $h,\dots,2h$.
  Per \Cref{cor:steenrod-flexible} this would amount to an increase of $\Theta(\log(n)^2)$ which is below leading order.
  Similarly, the exterior classes $h_{i,j}$ contribute at most $\Theta(\log(n)^2)$ since 
  a doubly infinite family of exterior generators behaves like a singly infinite family of polynomial generators
  \tqed
\end{rmk}

\begin{rmk}
  The arguments in this section can be simplified by using the maximal elementary quotients of $\A$ in place of $y(h)$.
  However, we would pay for this in \Cref{sec:telescope} where we use \Cref{lem:yn-size} a second time.
  \tqed
\end{rmk}

%% file: mfg.tex
Having computed the asymptotic growth of the Adams $E_2$-page,
in this section we turn to the Adams--Novikov $E_2$-page.


\begin{cnstr} \label{dfn:mfg-size}
  As in the previous sections we introduce several valuations,
  this time on $\check{\mathcal{D}}(\Mfg)$, to capture the quantities of interest:
  \begin{align*}
    \varpifg^c(X,\, n)
    &\coloneqq \mathrm{rank}_p\left( \oplus_{2u-s \leq n} \mathrm{H}^{s}(\Mfg;\, \omega^{\otimes u} \otimes X) \right),\\
    h_{\mathrm{fg}}(X,\, n)
    &\coloneqq \mathrm{rank}_p\left( \oplus_{2u-s = n} \mathrm{H}^{s}(\Mfg;\, \omega^{\otimes u} \otimes (\BP_*\F_p) \otimes X) \right),\\
    h_{\mathrm{fg}}^c(X,\, n)
    &\coloneqq \sum_{k \leq n} h_{\mathrm{fg}}(X,\, k).    
  \end{align*}
  
  We use $\varpifg^c(n)$ to measure the size of the cohomology of $\Mfg$ and $h_{\mathrm{fg}}$ to measure the cells used to build a sheaf.
  Much like with $\A$, since the cohomology vanishes for $t-s < 0$ we have a weight structure whose heart consists of shifts of the unit and objects which are almost compact can always be given a minimal cells structure.
  Similarly the analog of \Cref{ah-bound} holds with the same proof.
  \tqed
\end{cnstr}

\begin{thm} \label{thm:mfg-size}
  The cohomology of the moduli of formal groups grows at
  essentially the same rate as the cohomology of the Steenrod algebra.
  Specifically, we have
  \[ \log(\varpifg^c(n)) = \log(\varpi_{\A}^c(n)) + O(\log(n)^2). \]
  which implies the following pair of bounds:
  \begin{align}
    \log (\varpifg^c(n)) &\geq K_2 \log(n)^3 + O(\log\log(n)\log(n)^2), \\
    \log (\varpifg^c(n)) &\leq K_3 \log(n)^3 + O(\log\log(n)\log(n)^2).
  \end{align}
\end{thm}

We prove \Cref{thm:mfg-size} by comparing the moduli of formal groups with the moduli of height $\infty$ formal groups and then comparing the latter with the Steenrod algebra. We begin by recalling the standard (non-quantitative) version of these comparisons.

\begin{rec}
  The Thom reduction map factors through the height infinity locus in $\Mfg$ giving us maps of graded Hopf algebroids
  \begin{center}
    \begin{tikzcd}[row sep=tiny]
      (\BP_*,\, \BP_*\BP) \ar[r] & (\F_p,\, \F_p[t_1,\dots]) \ar[r] & (\F_p,\, \AA). \\
      \begin{matrix} \text{moduli of} \\ \text{formal groups} \end{matrix} & \begin{matrix} \text{moduli of height } \infty \\ \text{formal groups} \end{matrix} & \begin{matrix} \text{Steenrod} \\ \text{algebra} \end{matrix}
    \end{tikzcd}
  \end{center}  
  The middle term can be described as
  either the moduli of height $\infty$ formal groups or
  as the polynomial part of the dual Steenrod algebra, $\mathcal{P}$.
  The inclusion $\mathcal{P} \to \A$ extends to a short exact sequence of Hopf algebras
  \[ \mathcal{P} \to \AA \to \mathcal{E} \]
  where $\mathcal{E}$ is a graded Hopf algebra generated by primitive exterior classes $\tau_i$ in degrees of the form $2p^i - 1$.
  Associated to this short exact sequence we have a Cartan--Eilenberg sseq with signature
  \[ \H^{s,t}( \mathcal{P};\, \F_p[q_0,\dots] ) \cong {}^{\scaleto{\mathrm{CE}}{4pt}}\!E_2^{s,t,u} \Longrightarrow \H^{s+u,t}(\A) \]
  
  The first map is the inclusion of the closed substack obtained by killing the regular sequence $p, v_1, \dots$ in $\BP_*$.
  In particular, we have a $\BP_*\BP$-comodule algebra 
  $V(\infty) \coloneqq \BP_*/(p,v_1,\dots)$
  and a change of rings isomorphism
  \begin{align}
    \H^{s}(\Mfg; \omega^{\otimes u} \otimes V(\infty)) \cong \H^{s,2u}(\mathcal{P}) \label{eqn:vinf}
  \end{align}  
  The descent sseq associated to the map of $\BP_*\BP$-comodule algebras $\BP_* \to V(\infty)$ is the algebraic Novikov sseq which has signature
  \[ \H^{s,t}( \mathcal{P};\, \F_p[q_0,\dots] ) \cong {}^{\scaleto{\mathrm{alg.N}}{4.2pt}}\!E_2^{s,t,u} \Longrightarrow \H^{s,t}(\A). \]
  \tqed
\end{rec}

We now proceed to make these comparisons quantitative.

\begin{lem} \label{lem:mfg-lower}
  $ \log(\varpifg^c(V(\infty),\,n)) - \log(n) \leq \log(\varpifg^c(n)) $.
\end{lem}

\begin{proof}
  As a sheaf on $\Mfg$ the object $V(\infty)$ has cells which look like an exterior algebra on classes in degrees of the form $(-1, 2p^{n} - 2)$ i.e. in topological degree $2p^n -1$. Through dimension $n$ this exterior algebra has at most $n$ cells.
  Applying the analog of \Cref{ah-bound} we obtain the bound
  \[ \varpifg^c(V(\infty),\,n) \leq \h_{\mathrm{fg}}^c(V(\infty),\,n) \cdot \varpifg^c(n) \leq n \cdot \varpifg^c(n). \]
    
\end{proof}

\begin{lem} \label{lem:mfg-upper}
  $ \log(\varpifg^c(n)) \leq \log(\varpifg^c(V(\infty),\,n)) + O(\log(n)^2) $.
\end{lem}

\begin{proof}
  The algebraic Novikov sseq computes the cohomology of the moduli of formal groups from the cohomology of $\mathcal{P}$ and can be constructed to have an $E_0$-page given by
  \[ \F_p[q_0, q_1, \dots ] \otimes \H^{*,*}(\mathcal{P}) \]
  where $q_i$ corresponds to the class $v_i$ in degree $(0, 2p^n - 2)$.
  From this $E_0$-page we obtain an upper bound on $\varpifg^c(n)$ in terms of $\varpifg^c(V(\infty),\,n)$.
  This bound says that $\varpifg^c(n)$ is at most $\varpifg^c(V(\infty),\,n)$
  times the number of monomials in this polynomial algebra of degree at most $n$, which grows like $\exp(\Theta(\log(n)^2))$ by \Cref{cor:steenrod-flexible}. \qedhere
  
\end{proof}

The proofs that the cohomology of $\A$ and the cohomology of $\mathcal{P}$ grow at roughly the same rate are very similar to the proofs which appeared in Lemmas \ref{lem:mfg-lower} and \ref{lem:mfg-upper}.


\begin{lem} \label{lem:A-lower}
  $ \log(\varpifg^c(V(\infty),\,n)) - \log(n) \leq \log(\varpi_{\A}(n)) $.   
\end{lem}

\begin{proof}
  At odd primes the degeneration of the Cartan--Eilenberg sseq provides us with an injective map
  $ \H^{s,t}(\mathcal{P}) \to \H^{s,t}(\AA) $
  from which we obtain the desired lower bound.

  For $p=2$ we use the doubling isomorphism between $\AA$ and $\mathcal{P}$.
  This is an isomorphism
  $ \H^{s,t}(\AA) \cong \H^{s,2t}(\mathcal{P}) $
  from which we have
  \begin{align*}
    \varpifg^c(V(\infty),\,n)
    &= \mathrm{rank}_p\left( \oplus_{2u-s \leq n} \mathrm{H}^{s,2u}(\mathcal{P}) \right) \\
    &\leq n \cdot \mathrm{rank}_{\F_p[h_1]} \left( \oplus_{2u-s \leq n} \mathrm{H}^{s,2u}(\mathcal{P}) \right) \\
    &= n \cdot \mathrm{rank}_{\F_p[h_0]} \left( \oplus_{2u - s \leq n} \mathrm{H}^{s,u}(\A) \right) \\
    &\leq n \cdot \mathrm{rank}_{\F_p[h_0]} \left( \oplus_{u - s \leq n} \mathrm{H}^{s,u}(\A) \right) \\
    &= n \cdot \varpi_{\A}(n)
  \end{align*}
  \qedhere
  
    
\end{proof}

\begin{lem} \label{lem:A-upper}
  $ \log(\varpi_{\A}(n)) \leq \log(\varpifg^c(V(\infty),\,n)) + O(\log(n)^2) $.
\end{lem}

\begin{proof}
  The Cartan--Eilenberg sseq computes the cohomology of the Steenrod algebra from the cohomology of $\mathcal{P}$ and can be constructed to have an $E_0$-page given by
  \[ \F_p[q_0, q_1, \dots ] \otimes \H^{*,*}(\mathcal{P}) \]
  where $q_i$ is in degree $(1, 2p^n - 1)$.
  From this $E_0$-page we obtain an upper bound on $\varpi_{\A}(n)$ in terms of $\varpifg^c(V(\infty),\,n)$.
  This bound says that $\varpi_{\A}(n)$ is at most $\varpifg^c(V(\infty),\,n)$ times the number of monomials in this polynomial algebra of degree at most $n$ (after dropping $q_0$).
  By \Cref{cor:steenrod-flexible} the rank of this polynomial algebra grows like $\exp(\Theta(\log(n)^2))$.
\end{proof}

The first claim in \Cref{thm:mfg-size} follows from Lemmas \ref{lem:mfg-lower}, \ref{lem:mfg-upper}, \ref{lem:A-lower} and \ref{lem:A-upper}. The second claim is a reformulation of the first claim using the bounds from \Cref{thm:steenrod-precise}.

\begin{rmk}
  In our comparison between the Adams and Adams--Novikov $E_2$-pages we first removed all the $v_i$ and then put them back in again. This implies that the presence of the polynomial generators $v_i$ does not impact the leading order term in the growth rate of these $E_2$-pages. More succinctly: the size of any chromatic family is vanishingly small compared with the total number of chromatic families.
  \tqed
\end{rmk}


%% file: motivic.tex
In this section we leverage our understanding of the cohomology of the moduli of formal groups to estimate the size of the $p$-complete bigraded motivic stable stems over $\C$.
Remarkably, unlike in the classical case where only weak lower bounds are known, motivically we can prove our upper bounds are sharp.
Before proceeding we briefly review some motivic homotopy theory to fix notation.

\begin{rec} \label{rec:motivic}
  The category of motivic spectra, $\SHC$, is a stable presentably symmetric monoidal category containing a dualizable object $\Sigma^\infty X$ for each smooth pointed variety $X/\C$ with unit $\Ss$.
  In this category we have a distinguished copy of $\Z^{\oplus 2}$ in the picard group corresponding to the pair of invertible objects $\Sigma \Ss$ and $\Sigma^\infty \G_m$.
  Taking maps in from these picard elements provides us with \emph{bigraded} motivic homotopy groups
  \[ \pi_{s,w}X := \left[ \Sigma^{s-w} (\Sigma^\infty\mathbb{G}_m)^{\otimes w}, X \right] . \]

  We say that $X$ is cellular if it lives in the full subcategory generated under colimits by the bigraded spheres. 
  The structure of the $p$-completion of the cellular motivic category is best understood through the element $\tau \in \pi_{0, -1}\Ss_p$.
  This class is non-nilpotent and acts as a deformation parameter where the special fiber ($\tau = 0$) is algebraic and the generic fiber ($\tau = 1$) is the usual category of $p$-complete spectra.
  Making this precise, the cofiber of $\tau$ is a commutative algebra and the category of cellular $C\tau$-modules is equivalent to an appropriate category of sheaves on the moduli of formal groups \cite{ctau}.
  As a consequence we have an isomorphism of homotopy groups\footnote{As historical note, this isomorphism first appeared as \cite[Proposition 6.2.5]{StableStems}, prefiguring the categorical statement above.}
  \[ \pi_{s,w}^{\C}(C\tau) \cong \H^{2w-s, 2w}(\Mfg)_p. \]
  
  Morel connectivity tells us that $\pi^\C_{s,w}\Ss_p = 0$ for $s < w$ \cite{Morelconn}.
  As a corollary $\Ss_p$ is $\tau$-complete and therefore the $p$-complete motivic stable stems also vanish when $s < 0$. In the $p=2$ case these vanishing results are sharp because of the non-nilpotent elements $\tau \in \pi_{0,-1}^{\C}\Ss_2$ and $\eta \in \pi_{1,1}^{\C}\Ss_2$. 
  \tqed
\end{rec}

\begin{cnstr} \label{cnstr:mot-size}
  We package the the $p$-complete motivic stable stems into a valuation
  \[ \varpimot(X,\,n) \coloneqq \mathrm{rank}_{\Z_p[\tau]} \left( \oplus_{s \leq n} \pi_{s,*}^{\C}(X_p) \right). \]
  As usual we work with a cumulative measure of size in order to make rotation easier.
  \tqed
\end{cnstr}

\begin{thm} \label{thm:mot-size}
  The $p$-complete motivic stable stems grow at essentially the same rate as the cohomology of the moduli of formal groups in the sense that
  \[ \frac{1}{2} \varpifg^c(n) \leq \varpimot(n) \leq \varpifg^c(n). \]
  In particular, through our bounds on the size of the cohomology of the moduli of formal groups we have
  \begin{align}
    \log( \varpimot(n) ) &\geq K_2\log(n)^3 + O(\log\log(n)\log(n)^2), \\
    \log( \varpimot(n) ) &\leq K_3\log(n)^3 + O(\log\log(n)\log(n)^2).
  \end{align}  
\end{thm}

\begin{proof}
  Using subadditivity of valuations on the cofiber sequence
  \[ \Ss_p^{0,0} \to C\tau \to \Ss_p^{1,-1} \]
  we obtain the lower bound
  \[ \varpifg^c(n) = \varpimot ( C\tau ,\, n) \leq \varpimot(n) + \varpimot(\Ss_p^{1,-1},\, n) = \varpimot(n) + \varpimot(n-1) \leq 2\varpimot(n). \]
  The upper bound comes from the fact that modding out by $\tau$ and moving from considering rank over $\Z_p[\tau]$ to rank over $\Z_p$ cannot decrease rank unless there was a $\tau$-divisible element.
  As described in \Cref{rec:motivic}, Morel connectivity implies there cannot be $\tau$-divisible elements.
\end{proof}

Unlike in previous sections where the valuation we used was essentially the only reasonable choice, in $\SHC$ things are much more flexible. We end this section by discussing why although the choices made in defining $\varpimot$ were arbitrary they also had essentially no effect on the measured size of the motivic stable stems.

\begin{enumerate}
\item Before $p$-completion motivic homotopy groups over $\C$ are largely mysterious and often uncountably infinite.\footnote{This is a consequence of the rational $K$-theory of $\C$ being enormous.} This forces our hand in $p$-completing.
\item In defining $\varpimot$ we took the rank over $\Z_p[\tau]$.
  Although it might at first appear that the choice of algebra is important\footnote{For example $k[x,y]$ as a $k$-module might grow like $O(n)$, but grows like $O(1)$ as a $k[x]$-module} and could influence the perceived growth rate, this turns out to not be the case. The function $\exp(K \log(n)^3)$ grows sufficiently rapidly that considering rank over a larger, but still finitely generated, algebra will not change the leading order of the growth rate---an algebra with $k$ generators grows like $\exp(k \log(n))$ at the fastest.\footnote{It's worth mentioning that in the classical case there is another reason the choice of algebra wouldn't matter: Nishida nilpotence implies that every finitely generated algebra is finite.}
\item In defining $\varpimot$ we took a sum over $s \leq n$ in order to get a single variable function.
  We chose this line for its compatibility with Betti realization and so that there are no ``negative'' homotopy groups.
  Vanishing of negative groups holds when we sum along any line between the two extremal lines $s=0$ and $s=w$. 
  However, using the fact that the generators of the motivic stable stems as a $\Z_p[\tau]$-modules are concentrated in the wedge cut out by $2w \geq s$ and $s \geq w$ we see that a valuation produced by summing along a line will have the property that there is a constant $N$ such that the size is between $\varpimot(n/N)$ and $\varpimot(Nn)$.
  The function $\exp(K \log(n)^3)$ grows slowly enough that replacing $n$ by $n/N$ or $Nn$ has no effect on the leading order growth rate.\footnote{Note that this would not be the case if things were growing exponentially fast.}
\end{enumerate}

%% file: conj.tex
Although we could end the paper here,
we will instead press onward in pursuit of answers,
passing into the realm of conjecture and speculation.
Specifically, we now bring the telescope conjecture into our discussion. 
At this point most experts believe that the telescope conjecture,
which asserts that $T(h)$-localization and $K(h)$-localization are the same,
is unlikely to be true at heights $2$ and above.
Despite this an anti-telescope conjecture which provides a simple, conceptual reason for why this should be the case has yet to be formulated.
The best current understanding of telescopic homotopy comes from the the triple loop space approach of Mahowald, Ravenel and Schick \cite{MRS}.
In this approach they provide a collection of carefully chosen examples where the (expected) computational picture is relatively clear.

A key insight from the triple loop space approach
is that the telescope conjecture cannot fail by a small amount.
It must either be correct or fail violently.
In this section we seek to make this insight quantitative
by investigating the size of the telescopic stems.
To this end we begin by defining a function $\varpitel(X,\, n)$
which measures the size of the $T(h)$-local homotopy groups of
a type $h$ spectrum $X$.
Unlike in previous sections of this paper,
the parameter $n$ which usually represents the topological degree\footnote{As telescopic stems are \emph{periodic} there is little gain in thinking in terms of topological degree.}
instead represents depth in the localized Adams filtration.
Ideally we would then prove that as $X$ varies
$\varpitel$ provides a valuation on type $h$ spectra.
We were not able prove this.
Despite this we can still consider the measurements of telescopic homotopy provided by $\varpitel$.

In the second and third subsection we review the various conjectures of \cite{MRS} and
use them to compute the expected value of $\varpitel$ in certain examples.
Assuming that $\varpitel$ behaves as a valuation, we extrapolate from these cases to give a general conjecture on the value of $\varpitel$.

In the final subsection we return to the unlocalized world,
observing that $\varpitel$ cannot be large without
the stable homotopy groups of spheres also being large.
If our conjectures are correct,
this mechanism produces so many classes in the homotopy groups of spheres
that it saturates the upper bound from \Cref{sec:steenrod}.
This leads to the most surprising conclusion of this paper:
The size of the stable stems is dominated by the fate of the telescope conjecture.







\subsection{Measuring the size of telescopic stems}\ 




The localized Adams sseq, first introduced in \cite{Miller},
is a full-plane sseq with signature
\[ \AE_r^{s,t}(X)[q_h^{-1}] \Longrightarrow \pi_{t-s}(X[v_h^{-1}]), \]
constructed as a localization of the Adams sseq,
which converges to telescopic homotopy groups.
Note that in order to construct this localization $X$ must type $h$ and good (in the sense defined below\footnote{or we must move to a more modern perspective on the Adams sseq.}).


  


\begin{dfn}
  $X$ admits a \emph{good $v_h$-self map} if it has a self-map $v$ which induces a power of $v_h$ on $K(h)$-homology and also has a lift to a map acting along a line of slope $|v_h|^{-1}$ in the Adams sseq for $X$.
  \tqed
\end{dfn}


\begin{wrn}  
  There are many finite spectra which are not good.
  For example $\Cof(4)$ does not admit a good $v_1$-self map as
  its Adams $E_2$-page is just two copies of the $E_2$-page for the sphere.
  \tqed
\end{wrn}


The localized Adams sseq induces a filtration on the telescopic homotopy groups of a good type $h$ spectrum. By measuring the depth at which classes are born in this filtration we obtain a way of measuring the size of telescopic homotopy groups. 


\begin{dfn}  
  Suppose that $X$ is a type $h$ spectrum with a good $v_h$-self map $v$.
  Let $\mathrm{F}^s\pi_kX[v_n^{-1}]$ denote the filtration on $T(h)$-local homotopy groups of $X$ induced by the localized Adams sseq.
  Then we define
  \[ \varpitel(X,\, n) \coloneqq \frac{1}{|v|} \sum_{j=0}^{|v|-1} \log_p \left|  \mathrm{F}^{\frac{j}{|v_h|} - n}\pi_jX[v^{-1}] \right|. \]  
  \tqed
\end{dfn}

\begin{rmk}
  In the definition of $\varpitel$ we have measured the depth of classes below the line of slope $1/|v_h|$ on which $v$ acts and averaged over the periodicity of $X$.
  
  Alternatively we could have given a definition in terms of the sphere of origin, i.e. the minimal $n$ such that $x \in \pi_0X[v^{-1}]$ lifts to a class in $x \in \pi_nX$.
  As it turns out these definitions are commensurate (see \Cref{lem:tel-makes-classes}), so we have picked the one which we find easier to work with.
  \tqed
\end{rmk}

We would like to say that $\varpitel(-,\, n)$ extends to a valuation on the category of compact $T(h)$-local spectra and we conjecture that it (nearly) does.
There are three issues with this idea.

\begin{enumerate}
\item Not all compact type $h$ spectra admit good $v_h$-self maps.
\item We do not know whether every compact $T(h)$-local spectrum is the $T(h)$-localization of a compact type $h$-spectrum.
\item The function $\varpitel(-,\, n)$ is not obviously subadditive.
\end{enumerate}

The first two issues are not particularly difficult to resolve.
The subadditivity issue is much more substantial.
In a cofiber sequence $X \to Y \to Z$ if $Z$ has arbitrarily long Adams differentials,
then the Adams filtration of the image of a class from $X$ might jump upwards arbitrarily far, now being detected by a class on the $E_2$-page whose image in $Z$ is killed by a long differential.
Behavior of this sort breaks subadditivity (and certainly occurs).
However, there is a fix.
If we alter the target group of $\varpitel$ appropriately (by taking a quotient) then we can restore subadditivity.
For example if we assume that the localized Adams sseq collapses at a finite page, then it is enough to pass to the quotient by functions of the form $(1 - \sigma^k) \cdot \varpitel(X,\,n)$.
Unfortunately a theorem of this form is far out of reach at the moment.



\subsection{The triple loop space approach}\ 

The most direct strategy for disproving the telescope conjecture
is to find an $X$ for which the $T(h)$-local and $K(h)$-local homotopy groups can be computed (and then presumably, observed to be different).
The crux of this strategy lies in picking an $X$ for which these computations are as easy as possible without losing all information.

The triple loop space approach of \cite{MRS} focuses its attention on a sequence of $\E_1$-algebras $y(h)$ which we presently introduce.
Recall that the Hopkins--Mahowald theorem tells us that the Thom spectrum of the stable spherical bundle
$ \Omega^2 S^3 \to \mathrm{BGL}_1(\Ss_{(p)}) $
corresponding to the free double loop map on $(1+p) \in (\pi_0\Ss_{(p)})^{\times}$ is equivalent to $\F_p$.
Rewriting the source of this map as $\Omega(\Omega \Sigma) S^2$ we can consider the sequence of $\E_1$-algebras
\[ \Omega J_0 S^2 \to \Omega J_{p-1} S^2 \to \Omega J_{p^2-1} S^2 \to \cdots \to \Omega J_{p^h-1} S^2 \to \cdots \to \Omega J_\infty S^2 \]
coming from the James filtration\footnote{Note that we have restricted the filtration to integers of the form $(p^h-1)$.} on $\Omega \Sigma$.
Thomifying we obtain a diagram of $\E_1$-algebras
\[ \Ss \to y(1) \to y(2) \to \cdots \to y(h) \to \cdots \to \F_p. \]
These algebras have a number of convenient properties which make their telescopic homotopy groups particularly amenable to computation.
For the sake of brevity we will only discuss some of these properties\footnote{Somewhat humorously this means we have chosen to omit discussion of the \emph{triple loop} aspect of the triple loop space approach.}.

\begin{enumerate}
\item $y(h)$ is $T(j)$-acyclic for $j < h$ and has a polynomial subalgebra in its homotopy ring generated by classes $v_h, v_{h+1}, \dots, v_{2h}$.
  This gives us access to a simple formula for the $T(h)$-localization coming from inverting $v_h$.
  In other examples the first step in computing telescopic homotopy groups would be determining which self maps exist, which can be a rather delicate matter (see \cite{v1v2DavisMahowald} and \cite{BHHM} for example).
  Where convenient we will use $R(h)$ to denote the algebra $\F_p[v_h^{\pm 1}, v_{h+1}, \dots, v_{2h}]$. 
\item The Adams--Novikov sseq for $y(h)[v_h^{-1}]$ converges to the homotopy groups of the $K(h)$-localization.
  The $E_2$-page of this sseq is isomorphic to $R(h)$ tensored with the cohomology of a small congruence subgroup in the Morava stabilizer group.
  The congruence subgroup in question is in fact sufficiently small that its cohomology is isomorphic to that of its Lie algebra.
  Concretely we have 
  \[ \ANE_2^{s,t}(y(h)[v_h^{-1}]) \cong R(h) \left\langle h_{h+i,j}\ |\ 1 \leq i \leq h \text{ and } 0 \leq j \leq h-1 \right\rangle \]
where $|h_{h+i,j}| = (1,\ 2p^{h+i+j} - 2p^{j})$. \todo{check}  
\item Multiplication by the class $v_h$ from (1) is a good $v_h$-self map.
  The $E_2$-page of the localized Adams sseq is given by the $q_h$-localized cohomology described in
  \Cref{prop:mrs-e2}.
\end{enumerate}





From the nilpotence theorem we can deduce that every element outside $R(h)$ in the homotopy ring of $y(h)$ is nilpotent.
The following conjecture gives the expected pattern of differential which truncate the polynomial generators $b_{h+i,j}$ on the localized Adams $E_2$-page.

\begin{cnj}[{MRS differentials conjecture \cite[Conjecture 3.14]{MRS}}] \label{conj:mrs-phase-1}
  In the localized Adams sseq for $y(h)$ each class
  $h_{2h+i-j,j}$ survives to the $(2p^{j})^{\mathrm{th}}$-page,
  $b_{h+i,j}$ survives to the $(2p^{h-1}+1)^{\mathrm{st}}$-page
  and there are differentials
  \[ d_{2p^j}(h_{2h+i-j,j}) = q_h b_{h+i,h-1-j}^{p^j} \]
  for each $i > 0$ and $0 \leq j \leq h-1$.
\end{cnj}

The main result of \cite{MRSold} suggests that, conditional on the image of the map from the $T(2)$-local to the $K(2)$-local homotopy groups being as small as possible, the localized Adams sseq collapses at this point (i.e. at the $E_{2p}$-page).
On the other hand, \cite{MRS} is less committal, saying:

\begin{quote}
  [I]f in addition each $b_{n+i,j}$ were a permanent cycle, then we would have
  \begin{align}
    \label{eqn:yn-einf}
    E_\infty \cong R(n)_* \otimes E(h_{n+i,j}\ :\ i+j \leq n) \otimes P(b_{n+i,j})/(b_{n+i,j}^{p^{n-1-j}}).
  \end{align}
  For $n > 1$, this [...] is incompatible with the telescope conjecture.
  However, we cannot prove that each $b_{n+i,j}$ is a permanent cycle for $n>1$,
  and it seems unlikely to be true.
\end{quote}

There are two layers to the issue of whether $b_{h+i,j}$ is permanent.
We might have that instead of $b_{h+i,j}$ being permanent some other indecomposable class in that degree is permanent. This would be a relatively mild issue and wouldn't affect the considerations in this section. A more serious issue would be if no indecomposable class in this degree is permanent. Although we still lack a convincing argument that these classes should be permanent, in the twenty years since \cite{MRS} appeared, neither has any differential off of these classes been identified. For this reason we put forward the following conjecture:

\begin{cnj}[MRS collapse conjecture] \label{conj:mrs-collapse}  
  The classes $b_{h+i,j}$ (or similar indecomposable classes) are permanent cycles in the localized Adams sseq for $y(h)$.
\end{cnj}

The use of the localized Adams sseq in computing telescopic stems in the triple loop space approach is forced by the fact that it is currently the \emph{only} known method for computing telescopic stems.
In view of the fact that these methods presently seem inadequate to compute telescopic homotopy groups we offer the following challenge:

\begin{challenge}
  Find a (usable) device for computing telescopically localized homotopy groups which is not a repackaging of the localized Adams sseq.
\end{challenge}

\subsection{A sketch of telescopic homotopy theory}\

Using Conjectures \ref{conj:mrs-phase-1} and \ref{conj:mrs-collapse} we can compute $\varpitel(y(h),\,n)$ and from there we are in a position to make a simple numerical conjecture quantifying the expected failure of the telescope conjecture at any finite type $h$ spectrum. 

\begin{lem} \label{cor:yn-simple-size}
  Assuming Conjectures \ref{conj:mrs-phase-1} and \ref{conj:mrs-collapse}
  we have that
  \[ \log \varpitel(y(h),\, n) = \left( \binom{h}{2} + h \right) \log(n) + O(1). \]
\end{lem}

\begin{proof}
  With our assumptions we have 
  \[ \AE^{s,t}_\infty(y(h))[v_h^{-1}] \cong R(h)_* \otimes E(h_{h+i,j}\ :\ i+j \leq h) \otimes P(b_{h+i,j})/(b_{h+i,j}^{p^{h-1-j}}) \]
  from \cref{eqn:yn-einf}.
  The exterior classes only modify the size by $O(1)$ so we can ignore them.
  The depth in the Adams filtration of $b_{h+i,j}$ is within a constant multiple of $p^{h+i+j}$, therefore it suffices to compute the rank growth with these degrees (see \Cref{cor:steenrod-flexible}).
  If we cut each truncated polynomial generator up into a copy of $\F_p[x]/x^p$, then we see $\binom{h}{2}$ new generators in each degree of the form $p^N$ for large $N$.
  We can then reassemble these generators into $\binom{h}{2}$ polynomial generators.
  Adding on the other polynomial generators from $R(h)$ we have $\binom{h}{2} + h$ total polynomial generators.
  To conclude we observe that the cumulative rank of a graded polynomial algebra on $w$ generators grows like $C \cdot n^w$ asymptotically.
\end{proof}



In order to extract a reasonable conjecture for the behavior of $\varpitel$ on a compact type $h$ spectrum we observe that $y(h)$ has a number of cells that grows like a polynomial of degree $h$. This corresponds to the extra polynomial generators $v_{h+1}, \dots, v_{2h}$ which would not appear in a compact object.


As a corollary of the thick subcategory theorem of \cite{NilpII}
we know that $\langle X[v_h^{-1}] \rangle = \Sp_{T(h)}^\omega$ for every compact type $h$ spectrum $X$.
Morita invariance (up to a scalar) of size functions,
then suggests that the growth rate of
$\varpitel(X,\, n)$ should be essentially the same for every choice of
compact type $h$ spectrum $X$.
Taking the non-finiteness of $y(h)$ into account we obtain the following conjecture.

\begin{cnj} \label{conj:unipotent-growth}
  The telescopic homotopy groups of a finite type $h$ spectrum grow like 
  \[ \log \left( \varpitel(X,\, n) \right) = \binom{h}{2} \log(n) + O(1). \]
\end{cnj}


Following \cite[Theorem 1.5]{MRSold} the reader should not imagine these excess classes as building a polynomial algebra on $\binom{h}{2}$ generators but instead as a semi-perfect algebra whose tilt is of Krull dimension $\binom{h}{2}$. In the case $h=2$ the algebra we expect to see is $\F_p[t^{1/p^\infty}]/(t)$. The failure of the telescope conjecture then corresponds to the failure of descent for the map\footnote{The target is $\colim \Cof( t^{1/p^n} )$ and $\epsilon$ lives in degree $1$.}
\[ \F_p[t^{1/p^\infty}]/(t) \to \F_p\langle \epsilon \rangle. \]
Conjecturally, for $y(h)$ the image of the map from $T(h)$-local to $K(h)$-local homotopy contains the $h_{h+i,j}$ with $i+j \leq h$ while the remaining exterior classes fail to lift, instead appearing as semi-perfect generators.

We conclude by offering an outline,
extrapolated from Conjectures \ref{conj:mrs-phase-1} and \ref{conj:mrs-collapse},
of how we expect the localized Adams sseq of a compact type $h$ spectrum to behave.

\begin{enumerate}
\item There exists an $r \gg 0$ such that the upper bound provided by the $E_{r}$-page of the localized Adams sseq implies the upper bound in \Cref{conj:unipotent-growth}.
  The reader should compare this step with the MRS differentials conjecture and its amplification in the upper bound of \Cref{cor:yn-simple-size}.
\item There exists an $r \gg 0$ such that the localized Adams sseq for $X$ collapses at the $E_r$-page. The reader should compare this step with the MRS collapse conjecture.
\item At the point when the localized Adams sseq collapses there is a sufficient supply of surviving permanent cycles to satisfy the lower bound in \Cref{conj:unipotent-growth}.
  The reader should compare this step with the form of the $E_\infty$-page in the MRS collapse conjecture  and its amplification in the upper bound of \Cref{cor:yn-simple-size}.
\end{enumerate}  
 


\subsection{Back to the sphere}\ 

At this point we are finally ready to close the loop, arguing in the next theorem that
Conjectures \ref{conj:mrs-phase-1} and \ref{conj:mrs-collapse}
can be used to produce a sufficient supply of non-trivial classes in the homotopy groups of spheres to
saturate the upper bound provided by the Adams $E_2$-page.
This provides possibly the simplest and most striking consequence of the fate of the telescope conjecture.\footnote{The author was delighted to find that a question as natural as that of the size of the stable stems might be resolved through an understanding of telescopic homotopy theory.} 



\begin{thm} \label{thm:tel-to-sphere}
  Assuming Conjectures \ref{conj:mrs-phase-1} and \ref{conj:mrs-collapse} we have the following lower bound on the size of the stable stems
  \[ \log \left( \varpi(n) \right) \geq K_1 \log(n)^3 + O(\log\log(n)\log(n)^2) \]
  where $K_1$ is the constant given in \cref{eqn:constants}.
\end{thm}

The key point in the proof of this theorem is that every class in the telescopic homotopy groups has to come from some unlocalized class and we can use the localized Adams sseq to obtain bounds on how quickly these telescopic classes are born.
We make this formal in the next lemma.




\begin{lem} \label{lem:tel-makes-classes}
  Suppose that $X$ is a spectrum with a good $v_h$-self map $v$
  for which the localization map
  \[ \AE_r^{t,s}(X) \longrightarrow \left( \AE_r^{t,s}(X) \right)[q_h^{-1}] \]
  induces an isomorphism for $s \geq m(t-s) + c$ on the $E_2$-page.
  Then,
  \begin{enumerate}
  \item For every $r \leq \infty$
    the localization map induces a
    surjection for $s \geq m(t-s) + c$ and an
    isomorphism for $ s \geq m(t-s) + c + r $ on the $E_r$-page.
  \item The map
    $ \mathrm{F}^s\pi_kX \to \mathrm{F}^s\pi_kX[v_n^{-1}] $
    is surjective for $s \geq mk + c$ which implies that
    \[ \varpi(X,\, n) \geq |v| \cdot \varpitel \left( X,\, \left( m - \frac{1}{|v_n|} \right)(n-|v|) - c \right). \]
  \end{enumerate}  
\end{lem}


\begin{proof}
  The first claim follows by induction on $r$ together with the naturality of the map from the Adams sseq to the localized Adams sseq, which implies that a $v_n$-torsion class cannot support a differential killing a class which is non-trivial after $v_n$-localization.
  The first part of the second claim follows from the $r=\infty$ case of the first claim
  and the formula given is just a numerical expression of the implicit bound.
  %
  %
\end{proof}

\begin{rmk} \label{rmk:partial-partial}
  The idea behind 
  \Cref{lem:tel-makes-classes}
  can be pushed substantially further.
  For example,   
  suppose instead of having many classes on the $E_\infty$-page of the localized Adams sseq we knew only the weaker statement that there are many classes such that the lengths of the (potential) localized Adams differentials which kill them grow rapidly as a function of depth  in the localized Adams filtration. 
  Then, in the proof of \Cref{lem:tel-makes-classes} we could augment our argument with the information that a differential in the non-localized Adams sseq cannot have a source with $s$ negative. This would produce corresponding non-trivial classes in the homotopy groups of $X$, though not as many as produced by \Cref{lem:tel-makes-classes}. 
  \tqed
\end{rmk}

In effect \Cref{rmk:partial-partial} is saying that
\Cref{lem:tel-makes-classes}
can be made sufficiently flexible so that even partial progress towards disproving the telescope conjecture we will produce
large collections of classes in the homotopy groups of spheres.

\begin{lem} \label{lem:yn-conj-size}
  Assuming Conjectures \ref{conj:mrs-phase-1} and \ref{conj:mrs-collapse}
  we can use the telescopic homotopy groups of $y(h)$ to conclude that 
  the homotopy groups of $y(h)$ grow at least as quickly as a polynomial algebra on
  generators $v_{h+i}$ for $i=0,\dots,h$ in their usual degrees together with $j$ polynomial generators in degree $12p^{h+1+j}$ for $j = 1,\dots,n-1$.
\end{lem}

\begin{proof}
  From Conjectures \ref{conj:mrs-phase-1} and \ref{conj:mrs-collapse} we have that
  \[ \AE^{s,t}_\infty(y(h))[v_h^{-1}] \cong R(h)_* \otimes E(h_{h+i,j}\ :\ i+j \leq h) \otimes P(b_{h+i,j})/(b_{h+i,j}^{p^{h-1-j}}). \]
  In \Cref{lem:yn-size} we have already analyzed the degrees in which
  the localization map for $y(h)$ is an isomorphism on Adams $E_2$-pages.
  Using \Cref{lem:tel-makes-classes} and \Cref{lem:yn-size} we may therefore lift the classes
  $v_h^wb_{h+i,j}$ which live in topological degree
  \[ 12p^{h+1+i+j} - 10p^{1+i+j} - 2p^{1+j} - 2 \]
  to the un-localized homotopy of $y(h)$.
  We will use the subalgebra of the Adams $E_\infty$-page generated by these classes to bound the rank.
  Since increasing the degrees of the generators only decreases the cumulative rank we can bump the degree of $v_h^wb_{h+i,j}$ up to $12p^{h+1+i+j}$.
  At this point, as in \Cref{cor:yn-simple-size},
  the truncated generators behave the same as polynomial generators in the appropriate degrees.  
\end{proof}


\begin{proof}[Proof (of \Cref{thm:tel-to-sphere}).]
  This proof is very similar to the proof \Cref{thm:steenrod-precise}(2).
  \Cref{cor:AH} implies that
  \[ \varpi^c(y(h),\,n) \leq h^c(y(h),\,n) \cdot \varpi^c(n) \leq h^c(\F_p,\,n) \cdot \varpi^c(n) \]  
  Applying \Cref{cor:yn-einf-flexible} to the lower bounds on
  $\varpi(y(h),\,n)$ from \Cref{lem:yn-conj-size}
  we obtain 
  \[ K_1 \log(n)^3 + O(\log\log(n)\log(n)^2) \leq \log( \varpi^c(n)) + \log(h^c(\F_p)). \]
  \Cref{exm:steenrod-size} tells us that the error term $\log (h^c(\F_p))$ grows like $\Theta(\log(n)^2)$ which lets us conclude.
\end{proof}

\begin{rmk}
  Throughout this section we have made several conjectures which together provide a numerical picture of how we expect telescopic homotopy theory to behave.
  In various places these conjectures are likely too precise to be true as written.
  Although we hope history finds this specificity in speculation valuable, it is probably important to be flexible in interpretation.

  As a specific example:
  In \Cref{conj:mrs-phase-1} the stated differentials should not be taken literally but instead up to some filtration so that Miller's computation of the localized Adams sseq for a Moore spectrum \cite{Miller} is sufficient to say this conjecture is true for height 1 at odd primes\footnote{This brings up an interesting point, at height 1 although we know the $E_3$-page of the localized Adams sseq for a Moore spectrum, we do not actually have precise knowledge of the $d_2$ differentials. It would seem desirable to have refinement of Miller's result which provides a precise formula for these differentials.}.
    
  Above all, the reader is encouraged to treat this section as an invitation to further investigation and correction.
\tqed \end{rmk}
  


%% file: unstable.tex
\def\CU{\mathrm{A}}
\def\U{\mathcal{U}}
\def\V{\mathcal{V}}
\def\CA{\mathcal{CA}}
\DeclarePairedDelimiter\abs{\lvert}{\rvert}%
\makeatletter
\let\oldabs\abs
\def\abs{\@ifstar{\oldabs}{\oldabs*}}

In this appendix, we consider bounds on the rank and torsion of unstable homotopy groups.
We begin with the ranks of unstable homotopy groups, which we study in \Cref{sec:unstab-rk-bound}.
Given a simply-connected space $X$ of finite type, in \Cref{thm:unstab-bd} we establish a bound on $\rank_p (\pi_n (X))$ in terms of $\rank_{\F_p} (\H_n (\Omega X; \F_p))$ and $\rank_{\F_p [q_0]} (\Ext_{\mathcal{A}} ^{s,t} (\F_p, \F_p))$.
When combined with \Cref{thm:steenrod-precise}, this implies that the rank of the unstable homotopy groups of a sphere grows no faster than $\exp(O(\log(n)^3))$.
This is the first known subexponential bound on the rank of unstable homotopy groups of spheres.
The question of whether such a bound exists has been recently highlighted by Huang and Wu \cite[Question 1.7]{Huang-Wu}.
We are moreover able to verify a conjecture of Henn \cite[Conjecture on p. 237]{Henn-bound}, which states that the order of exponential growth of $\rank_p (\pi_n (X))$ and $\rank_{\F_p} (\H_n(\Omega X; \F_p))$ is equal.


What we have to say about torsion bounds, a subject with a rich history, is less original.
In contrast to the Cohen--Moore--Neisendorfer theorem, which provides an optimal bound on the $p$-torsion exponent of the unstable homotopy groups of spheres for odd $p$, we will focus in \Cref{sec:un-tor-bd} on results which apply to a general class of spaces.
We begin by recalling the current state-of-the-art, which is due to Barratt \cite{Barratt} and determines a bound on the $p$-torsion order of a suspension.
We then prove \Cref{thm:no-suspension-tor-bound}, which is a simple application of the Goodwillie calculus to unstable torsion bounds.

\subsection{Bounds on the rank}\label{sec:unstab-rk-bound}\ 

The main theorem of this section will be phrased in terms of certain generating series, which we define below.

\begin{dfn}
  Let $X$ denote a based space, $M$ denote a connective $\A$-comodule and $V$ denote a non-negatively graded $\F_p$-vector space. Then we define generating series:
  \begin{align*}
    \varpi(X;t) &\coloneqq \sum_{k \geq 0} \rank_p(\pi_k X) \cdot t^k, &
    h(X;t) &\coloneqq \sum_{k \geq 0} \rank_{\F_p} ( \H_k (X;\F_p)) \cdot t^k, \\
    \varpi_\A (M;t) &\coloneqq \sum_{k \geq 0} \rank_{\F_p [q_0]} \left( \bigoplus_{u-s = k} \Ext^{s,u} _{\A} (M) \right) \cdot t^k, &
    P(V ;t) &\coloneqq \sum_{k \geq 0} \rank_{\F_p} (V_k) \cdot t^k.
  \end{align*}
  In our definition of $\varpi_\A (M;t)$, $\F_p [q_0] \subseteq \Ext^{s,t} _{\A} (\F_p)$ is the polynomial subalgebra generated by the class $q_0 \in \Ext^{1,1} _{\A} (\F_p)$ detecting $p \in \pi _0 (\Ss)$.
  When $M = \F_p$, we will use the abbreviated notation $\varpi_\A (t) \coloneqq \varpi_\A (\F_p, t)$ where convenient.
  \tqed
\end{dfn}

Our main theorem bounds the rank of $\pi_* (X)$ in terms of the rank of $\H_*(\Omega X; \F_p)$:

\begin{thm} \label{thm:unstab-bd}
  Let $X$ denote a simply-connected space of finite type. Then we have a bound:
  \[\varpi(X;t) \cdot t^{-1} = \varpi(\Omega X;t) \leq 2 \cdot P(\A; t) \cdot \varpi_\A (t) \cdot h(\Omega X; t). \]
  Combining this bound with \Cref{thm:steenrod-precise} and \Cref{cor:steenrod-flexible}, we obtain the bound:
  \[\log\left(\rank_p( \pi_n X )\right) \leq \log \left( \rank_{\F_p} (\oplus_{i \leq n-1} \H_{i} (\Omega X; \F_p)) \right) + O(\log(n)^3).\]
\end{thm}

Since $\pi_* (S^k)_{(p)}$ is of finite exponent \cite[Theorem 8.10]{TTor} \cite[Corollary 1.22]{JTor}, we may apply \Cref{thm:unstab-bd} to the case $X = S^k$ to obtain the following corollary:

\begin{cor}
  There is a bound
  \[\log_p ( \abs{\pi_n (S^k)_{(p)}} ) \leq \exp(O(\log(n)^3)).\]
\end{cor}

As mentioned above, this is the first known subexponential bound on the size of the unstable homotopy groups of spheres.
More generally, \Cref{thm:unstab-bd} implies that the exponential growth of $\rank_p( \pi_n (X)_{(p)})$ is controlled by the exponential growth of $\rank_{\F_p} (\H_{n} (\Omega X; \F_p))$, which is easier to understand.
As a consequence, we are able to deduce a conjecture of Henn \cite[Conjecture on p. 237]{Henn-bound}.

\begin{ntn}
  Given a formal power series $F(t) \in \mathbb{R}[\![t]\!]$, let $R_{F(t)} \in [0,\infty]$ denote the radius of convergence of $F(t)$.
  \tqed
\end{ntn}

\begin{cor}[{{\cite[Conjecture on p. 237]{Henn-bound}}}]\label{cor:Henn}
  Let $X$ denote a simply-connected space of finite type. Then
  \[\min(1, R_{\varpi(X;t)}) = \min(1, R_{h(\Omega X; t)}).\]
  Moreover, if $X$ is a simply-connected finite space, then
  \[R_{\varpi(X;t)} = R_{h(\Omega X;t)}.\]
\end{cor}

\begin{proof}
  The inequality
  \[\min(1, R_{\varpi(X;t)}) \leq \min(1, R_{h(\Omega X; t)})\]
  was proven by Iriye \cite{Iriye}, so it suffices to prove the opposite inequality.
  This is a consequence of \Cref{thm:unstab-bd}, \Cref{thm:steenrod-precise} and \Cref{cor:steenrod-flexible}.
%
%

  If $X$ is simply-connected and finite, then $\pi_n (X)_{(p)}$ is nonzero for infinitely many $n$ by a theorem of Serre and Umeda \cite{SerreEM, Umeda}. Similarly, it follows from \cite[Lemma 5.12]{DW} that $\H_n (\Omega X; \F_p)$ must be nonzero for infinitely many $n$.
  Thus $R_{\varpi(X;t)} \leq 1$ and $R_{h(\Omega X; t)}\leq 1$, so we may conclude.
\end{proof}

\begin{rmk}
  Henn works with the generating series for $\rank_p (\pi_* (X; \Z/p))$ instead of $\rank_p (\pi_* (X))$. As far as radii of convergence go, this makes no difference, since $\rank_p (\pi_{i} (X)) \leq \rank_p (\pi_i (X; \Z/p)) \leq \rank_p (\pi_i(X)) + \rank_p (\pi_{i-1} (X))$.
  \tqed
\end{rmk}

\begin{rmk}\label{rmk:Henn-result}
  In \cite[Theorem 1]{Henn-bound}, Henn proved the following weakened version of \Cref{cor:Henn}. Given a simply-connected space $X$ of finite type, there is an inequality
  \[\min(1/2,R_{h (\Omega X;t)}) \leq R_{\varpi (X;t)}.\]
  Boyde has improved this to the inequality \cite[Corollary 1.3]{Boyde}:
  \[\min((1/2)^{\frac{1}{p-1}},R_{h (\Omega X;t)}) \leq R_{\varpi (X;t)}.\]
  \tqed
\end{rmk}

Our proof of \Cref{thm:unstab-bd} follows the strategy Henn used to prove \cite[Theorem 1]{Henn-bound}.
As a first step, Henn combines the unstable Adams spectral sequence of $\Omega X$ with a Grothendieck spectral sequence studied by Miller to reduce to proving a bound on the size of Ext groups of unstable $\A$-modules.
To obtain this estimate, Henn utilizes an inductive argument which is based on the algebraic EHP sequence and uses the known unstable Ext groups of $\Sigma \F_p$ as a base case.
Our argument will follow a similar approach, with the key difference being that we use the bound of \Cref{thm:steenrod-precise} on the size of Ext groups of stable $\A$-modules as our base case.
In other words, we use the algebraic EHP sequence in the opposite direction as Henn:
while he uses it to propagate the unstable Ext of $\Sigma \F_p$ upwards, we use it to propagate the stable Ext groups downward.

\subsection{A recollection on the unstable Adams sseq}\ 

We begin by recalling the spectral sequences we will use, as well as the algebraic categories which govern their $\mathrm{E}_2$-pages. We start with the unstable Adams spectral sequence.

\begin{dfn}
  Let $\CA$ denote the category of connected unstable coalgebras without unit over the dual Steenrod algebra---see e.g. \cite{Sullivan} for the full definition.
  \tqed
\end{dfn}

The unstable Adams spectral sequence for $\Omega X$ takes the form 
\[\Ext^{s} _{\CA} (\Sigma^t \F_p, \overline{\H}_* (\Omega X; \F_p)) \Rightarrow \pi_{t-s} (\Omega X)^{\wedge} _{p}.\]
Since $X$ is simply-connected and of finite type, it follows that $\Omega X$ is connected, nilpotent and of finite type, so that the spectral sequence converges.

To bound the $\mathrm{E}_2$-page of the unstable Adams spectral sequence, we will make use of Miller's Grothendieck spectral sequence for Ext groups in $\CA$, which was constructed in \cite{Sullivan,SullivanCor}.\footnote{Note that the treatment of this spectral sequence in the odd-primary case was incorrect in \cite{Sullivan}---the category $\V$ below was introduced by Miller in the correction \cite{SullivanCor}.}
This spectral sequence is based on certain categories $\U$ (at the prime $2$) and $\V$ (at odd primes) of unstable comodules over the dual Steenrod algebra.

\begin{dfn}
  At the prime $2$, let $\U$ denote the category of comodules over the dual Steenrod algebra whose induced right action of the Steenrod algebra satisfies
  \[x \Sq^n = 0 \text{ if } \abs{x} \leq 2n-1.\]
  At an odd prime $p$, let $\V$ denote the category of comodules over the dual Steenrod algebra whose induced right action of the Steenrod algebra satisfies
  \[x \mathrm{P}^n = 0 \text{ if } \abs{x} \leq 2pn.\]
  \tqed
\end{dfn}

Given $C \in \CA$, let $R^t P (C)$ denote the right derived functors of the coalgebra primitives functor.
Then $\Sigma^{-1} R^t P (C) \in \U$ at the prime $2$, and $\Sigma^{-1} R^t P(C) \in \V$ at odd primes.
Miller's spectral sequences then take the form:
\[\Ext^s _{\U} (M, \Sigma^{-1} R^t P (C)) \Rightarrow \Ext^{s+t} _{\CA} (\Sigma M, C)\]
for $M \in \U$ and $C \in \CA$ at the prime $2$, and
\[\Ext^s _{\V} (N, \Sigma^{-1} R^t P (C)) \Rightarrow \Ext^{s+t} _{\CA} (\Sigma N, C)\]
for $N \in \V$ and $C \in \CA$ at odd primes $p$.

\begin{rmk}
  Following Henn \cite{Henn-bound}, the reason that we work with the unstable Adams spectral sequence of $\Omega X$ instead of that of $X$ is that the Milnor--Moore theorem \cite[Theorem 7.11]{MM} allows us to deduce that Miller's spectral sequence for $\Omega X$ is concentrated in degrees $t=0,1$.
  \tqed
\end{rmk}

To bound Ext groups in the categories $\U$ and $\V$, we will combine the stable bound proven in \Cref{thm:steenrod-precise} with the algebraic EHP sequence, which we will now recall.
The 2-primary version of this sequence appears in \cite{Henn-bound} after the proof of Lemma 3, and the odd-primary version appears as 3(b) on p.143 of \cite{Henn-bound}.
Let $M(p)$ denote the $\A$-comodule which is equal to $\F_p$ in degrees $0$ and $1$ with a Bockstein connecting them.

\begin{ntn}
  Given $M \in \U$, we let
  $\Ext^{s,t}_{\mathcal{A}} (M) \coloneqq \Ext^s _{\mathcal{A}} (\Sigma^t \F_2, M).$
  Given $N \in \V$, we define $\Ext^{s,t} _{\mathcal{A}} (N)$ similarly.
\end{ntn}

\begin{thm}[Algebraic EHP]
  At the prime $2$, there are long exact sequences:
%
%
%
  \[\cdots \xrightarrow{\H} \Ext^{s-2,t}_{\U} (\Sigma^{2k+1} \F_2) \xrightarrow{P} \Ext^{s,t} _{\U} (\Sigma^k \F_2) \xrightarrow{E} \Ext^{s,t+1}_{\U} (\Sigma^{k+1} \F_2) \xrightarrow{\H} \Ext^{s-1,t}_{\U} (\Sigma^{2k+1} \F_2) \xrightarrow{P} \cdots.\]
  On the other hand, at odd primes $p$ there are equivalences
  \[\Ext_\V ^{s,t} (\Sigma^{2k-1} \F_p) \cong \Ext_\V ^{s,t+1} (\Sigma^{2k} \F_p) \]
  and long exact sequences
  \[\cdots \xrightarrow{\H} \Ext^{s-2,t}_\V (\Sigma^{kp} M(p)) \xrightarrow{P} \Ext^{s,t}_\V (\Sigma^{2k} \F_p) \xrightarrow{E} \Ext^{s,t+1}_\V (\Sigma^{2k+1} \F_p) \xrightarrow{\H} \Ext^{s-1,t}_\V (\Sigma^{kp} M(p)) \xrightarrow{P} \cdots.\]
\end{thm}

\subsection{The proof of \Cref{thm:unstab-bd}}\ 

We will now put together the above tools to prove \Cref{thm:unstab-bd}.
First, we need some notation for the generating series of Ext groups in $\U$ and $\V$.

\begin{dfn}
  Given an object $M$ of $\U$, we let $\varpi_\U (M;t)$ denote
  \[\varpi_\U (M;t) \coloneqq \sum_{k \geq 0} \rank_{\F_p [q_0]} \left( \bigoplus_{u-s = k} \Ext^{s,u} _{\U} (M) \right) \cdot t^k.\]
  We define $\varpi_\V (M,t)$ analogously for objects $M$ of $\V$.
  When $M = \Sigma^n \F_p$, we abbreviate to $\varpi_\U (n,t)$ and $\varpi_\V (n,t)$.
  \tqed
\end{dfn}

We begin by recording some consequences of the unstable Adams spectral sequence for $\Omega X$ and Miller's Grothendieck spectral sequences.

\begin{prop} \label{prop:sseq-consq}
  Given a simply-connected space $X$ of finite type, there is an inequality
  \[\varpi(\Omega X; t) \leq \varpi_{\U} (P (\overline{\H}_{*} (\Omega X));t) + \varpi_{\U} (R^1 P (\overline{\H}_{*} (\Omega X));t) \cdot t^{-1}\]
  at the prime $2$, and an inequality
  \[\varpi(\Omega X; t) \leq \varpi_{\V} (P (\overline{\H}_{*} (\Omega X));t) + \varpi_{\V} (R^1 P (\overline{\H}_{*} (\Omega X));t) \cdot t^{-1}\]
  at odd primes $p$.
  Furthermore, there are inequalities
  \[P(P (\overline{\H}_{*} (\Omega X));t) \leq P(\overline{\H}_* (\Omega X),t)\]
  and
  \[P( R^1 P (\overline{\H}_{*} (\Omega X));t) \leq P(\overline{\H}_* (\Omega X),t) \cdot t.\]
\end{prop}

\begin{proof}
  Making use of the unstable Adams spectral sequence for $\Omega X$ and Miller's Grothendieck spectral sequences, the first half of the proposition reduces to the claim that $R^t P(\overline{\H}_{*} (\Omega X)) = 0$ for $t \geq 2$.

  It follows from \cite[Theorem 7.11]{MM} that, as a coalgebra, $\H_* (\Omega X)$ is a tensor product of divided power coalgebras, exterior coalgebras, and duals of truncated polynomial algebras.
  The proposition then follows from the standard computation of the derived primitives of these coalgebras.
\end{proof}

%
%
%
%

Our goal is now to bound $\varpi_\U (M;t)$ and $\varpi_\V (M;t)$ in terms of stable Ext over the dual Steenrod algebra. For this, we make use of the algebraic EHP sequence, which has the following immediate corollary:

\begin{cor}
  At the prime $2$, we have
  \[\varpi_\U (n;t) \leq \varpi_\U (n+1; t) \cdot t^{-1} + \varpi_\U (2n+1; t)\cdot t^{-2}.\]
  At odd primes, we have
  \[\varpi_{\V} (2k-1;t) = \varpi_{\V} (2k;t)\cdot t^{-1}\]
  and
  \[\varpi_{\V} (2k;t) \leq \varpi_\V (2k+1; t) \cdot t^{-1} + \varpi_\V (2pn;t) \cdot t^{-2} + \varpi_\V (2pn+1;t) \cdot t^{-2}\]
\end{cor}

We now introduce certain integer sequences which control the combinatorics of the above recurrences.

\begin{dfn}
  Let 
  \[I (n) = \{(i_1, \dots, i_k) \vert k \geq 0 , \. i_k \geq n \text{ and } i_s > 2 i_{s+1}\}.\]
  Given a sequence $J = (i_1, \dots i_k) \in I(n)$, we let $\dim(J) = \sum_{s=1} ^k (i_s -1)$.

  At odd primes, let
  \[I(n) = \{(\epsilon_1, i_1, \dots, \epsilon_k, i_k) \vert k \geq 0 ,  \epsilon_s \in \{0,1\},  2i_k \geq n \text{ and } i_s > p i_{s+1} - \epsilon_{s+1}\}.\]
  Given $J = (\epsilon_1, i_1, \dots, \epsilon_k, i_k) \in I(n)$, we set $\dim(J) = \sum_{s=1} ^k (2(p-1)i_s-\epsilon_s-1)$.

  We now set
  \[\CU(n;t) = \sum_{J \in I(n)} t^{\dim(J)}.\]
  \tqed
\end{dfn}

\begin{rmk}
  The sequences in $I(n)$ are the completely unadmissible sequences of excess $n$ considered by Behrens in \cite{BehrensEHP}.
  They correspond to operations of Dyer--Lashof-type on spectral Lie algebras.
  In this guise they play a role in the Goodwillie tower of the identity.
  We suspect that the bound which we prove in \Cref{prop:good-bd} may also be obtained from an as-yet unstudied algebraic Goodwillie spectral sequence.
  \tqed
\end{rmk}

The following lemma is immediately verified from the definitions:

\begin{lem}
At the prime $2$, we have the following relation:
\[\CU(n;t) = \CU(n+1;t) + \CU(2n+1; t)\cdot t^{n-1}.\]
On the other hand, at odd primes we have:
\[\CU(2n-1;t) = \CU(2n;t)\]
and
\[\CU(2n;t) = \CU(2n+1;t) + \CU (2pn; t) \cdot t^{2(p-1)n-2} + \CU(2pn+1;t) \cdot t^{2(p-1)n-1}.\]
\end{lem}

We are now able to give a bound on the size of unstable Ext groups in terms of $\CU(n;t)$ and stable Ext groups.

\begin{prop} \label{prop:good-bd}
  We have
  \[\varpi_{\U} (n;t) \leq \CU(n;t) \cdot \varpi_\A (t) \cdot t^n\]
  and
  \[\varpi_{\V} (n;t) \leq \CU(n;t) \cdot \varpi_\A (t) \cdot t^n.\]
\end{prop}

\begin{proof}
  We assume that $p=2$---the case of $p$ odd is similar.
  First, note that
  \[\varpi_{\U} (n;t) \equiv \varpi_{\A} (t) \cdot t^n \equiv \CU(n;t) \cdot \varpi_\A (t) \cdot t^n \mod t^{2n-1}\]
  by the stabilization of unstable Ext.\footnote{This may be read off from the algebraic EHP sequence and the fact that $\varpi_{\U} (n;t) \equiv 0 \mod t^{n-1}$, for example.}
  We will show that 
  \begin{align} \label{eq:app}
    \varpi_{\U} (n;t) \leq \CU(n;t) \cdot \varpi_\A (t) \cdot t^n \mod t^{2n+k}
  \end{align}
  for all $n$ and $k$ by induction on $k$.
  It follows from the above that (\ref{eq:app}) holds for $k=-1$, which we take as our base case.

  Assume that (\ref{eq:app}) holds for $k-1$.
  We have the inequality
  \[\varpi_\U (n;t) \leq \varpi_\U (n+1; t) \cdot t^{-1} + \varpi_\U (2n+1; t)\cdot t^{-2}.\]
  and the equality
  \[\CU(n;t) \cdot \varpi_{\A} (t) \cdot t^n = \CU(n+1;t) \cdot \varpi_{\A} (t) \cdot t^n + \CU(2n+1; t) \cdot \varpi_{\A} (t) \cdot t^{2n-1}.\]
  The inequalities
  \[\varpi_{\U} (n+1;t) \cdot t^{-1} \leq \CU(n+1;t) \cdot \varpi_\A (t) \cdot t^{n+1} \cdot t^{-1} \mod t^{2n+k}\]
  and
  \[\varpi_{\U} (2n+1;t) \cdot t^{-2} \leq \CU(2n+1;t) \cdot \varpi_\A (t) \cdot t^{2n+1} \cdot t^{-2} \mod t^{2n+k}\]
  hold by the induction hypothesis, so we are done.
\end{proof}

We next record a convenient bound on $\CU(n;t)$ which is uniform in $n$.

\begin{lem} \label{lem:cu-bd}
  We have
  \[\CU(n;t) \leq P(\A; t)\]
  for $n \geq 2$ at the prime $2$, and for $n \geq 3$ at odd primes.
\end{lem}

\begin{proof}
  At the prime $2$, sending $J=(i_1, \dots, i_k)$ to $J-1= (i_1-1, \dots, i_k-1)$ gives a bijection between $I(n)$ and the set $I'(n) = \{(i_1, \dots, i_k) \vert k \geq 0 , \. i_k \geq n-1 \text{ and } i_s > 2 i_{s+1}+1\}$ for which $\dim(J) = \abs{J-1} \coloneqq \sum_s (i_s-1)$.

  A basis for $\A$ is given by the duals of products $\Sq^{i_1}\cdots \Sq^{i_k}$ of Steenrod squares indexed by the admissible sequences $(i_1, \dots, i_k)$ with $i_s \geq 2 i_{s+1}$.
  The element $\Sq^{i_1}\cdots \Sq^{i_k}$ lies in grading $\sum_s i_s$.
  The result now follows from the fact that $I'(n)$ is a subset of the admissible monomials when $n \geq 2$.

  The proof at odd primes is similar.
\end{proof}

We are now ready to state and prove our  main bounds on $\varpi_{\U} (M;t)$ and $\varpi_{\V} (M;t)$.

\begin{cor} \label{cor:un-ext-m-bd}
  We have
  \[\varpi_{\U} (M;t) \leq P(\A;t) \cdot \varpi_\A (t) \cdot P(M;t)\]
  for any connected $M \in \U$.
  Similarly, we have
  \[\varpi_{\V} (M;t) \leq P(\A;t) \cdot \varpi_\A (t) \cdot P(M;t)\]
  for any connected $M \in \V$.
\end{cor}

\begin{proof}
  Consider the spectral sequence associated to the filtration of $M$ by degree.
  The associated graded is equal to $M$ as a graded $\F_p$-module but has trivial Steenrod action.
  This reduces us to the case where $M = \Sigma^n \F_p$.
  When $n \geq 2$ (at $p=2$) or $n \geq 3$ (at odd $p$), the desired bound follows from \Cref{prop:good-bd} and \Cref{lem:cu-bd}.
  On the other hand, we have
  \[\Ext_{\U} ^{s,t} (\Sigma \F_2) \cong \F_2 [q_0]\]
  and
  \[\Ext_{\V} ^{s,t} (\Sigma \F_p) \cong \Ext_{\V} ^{s,t+1} (\Sigma^{2} \F_p) \cong \F_p [q_0].\]
  which are clearly smaller than the stable Ext.
\end{proof}

Putting everything together, we are finally able to prove the main theorem:

\begin{proof}[Proof of \Cref{thm:unstab-bd}]
  Combine \Cref{prop:sseq-consq} with \Cref{cor:un-ext-m-bd}.
%
\end{proof}

\subsection{Torsion bounds}\label{sec:un-tor-bd}\ 

On a random space $X$ we expect that there are infinitely many non-trivial rational homotopy groups.
Thus, in order to have a discussion of torsion exponents we must restrict the class of spaces under consideration.
A natural collection of spaces we might work with are those which are finite, simply-connected and rationally trivial.
Note that in the stable case this is enough for there to be a \emph{uniform} torsion bound.
However, unstably things are not known to be this simple.
The present state-of-the-art is the following theorem of Barratt.

\begin{thm}[Barratt \cite{Barratt}] \label{thm:Barratt}
  Suppose we are given an $(s-1)$-connected space $X$ with $s \geq 1$
  such that the identity map $\Sigma X \to \Sigma X$ is $p^m$-torsion.
  Then $p^{mk} \cdot \pi_n(\Sigma X) = 0$ for $n \leq 2^k s$.
  Moreover, if $X$ is itself a suspension, then $p^{m+k} \cdot \pi_n(\Sigma X) = 0$ for $n \leq p^{k+1}s$.
\end{thm}

For spaces which are suspensions this theorem is quite powerful as it implies that the torsion order grows logarithmically as a function of $n$.
On the subject of a general space Barratt asserts a bound of the order $p^{2^{mk}}$ on $\pi_n(X)$ for $n \leq 2^k(n-1)$ \cite[p.125]{Barratt} and proves that some sort of bound must exist \cite[Lemma 1.3]{Barratt}. Unfortunately the present authors were unable to determine which argument Barratt had in mind in this remark. For this reason we include a proof of a bound of a similar sort.

\begin{thm} \label{thm:no-suspension-tor-bound}
  If $X$ is $s$-connected with $s \geq 1$
  and the identity map on $\Sigma^\infty X$ has order $p^m$, then
  \[ \tors_p\left( \pi_n(X) \right) \leq \frac{m+1}{s} \cdot n. \]
\end{thm}

\begin{proof}
  We prove the theorem by analyzing the Goodwillie tower of the identity on the category of spaces evaluated at $X$.
  Since $X$ is $s$-connected and $s \geq 1$,
  this is a convergent tower of principal fibrations
  \begin{center}
    \begin{tikzcd}[sep=small]
      & \Omega^\infty (\partial_3(\Id) \otimes (\Sigma^\infty X)^{\otimes 3})_{h\Sigma_3} \ar[d] & \Omega^\infty (\partial_2(\Id) \otimes (\Sigma^\infty X)^{\otimes 2})_{h\Sigma_2} \ar[d] & \Omega^\infty  \Sigma^\infty X \ar[d, equal] \\
      X \to \cdots \ar[r] & P_3(X) \ar[r] & P_2(X) \ar[r] & P_1(X)
    \end{tikzcd}
  \end{center}
  where the term $\Omega^\infty (\partial_k(\Id) \otimes (\Sigma^\infty X)^{\otimes k})_{h\Sigma_k}$
  is $ks$-connected.
  In \Cref{lem:norm-order} we will show that the identity map on $(\Sigma^\infty X)^{\otimes k}$ in $\Sp^{\Sigma_k}$ has order $p^{m+v_p(k)}$.
  This implies that
  \[ \tors_p\left( \pi_n(X) \right) \leq \sum_{0 < sk < n} m + v_p(k) \leq (m+1) \frac{n}{s}. \]
\end{proof}

\begin{lem} \label{lem:norm-order}
  If a map $f : X \to Y$ of spectra is $p^m$-torsion,
  then the $\Sigma_n$-equivariant\footnote{In this lemma we work Borel equivariantly.} map $f^{\otimes n}$ is $p^{m+v_p(n)}$-torsion.
\end{lem}

\begin{proof}
  We begin by observing that the $\Sigma_n$-equivariant composite
  \[ \Ss \to \Ss[\Sigma_n/W] \to \Ss \]
  where $W$ is a $p$-Sylow subgroup
  is multiplication by $|\Sigma_n/W|$ on underlying and
  therefore a $p$-local equivalence in $\Sp^{\Sigma_n}$.
  Thus, it suffices to prove that $f^{\otimes n}$ is $W$-equivariantly $p^{m+v_p(n)}$-torsion.
    
  The group $W$ is a product $W_1^{\times a_1} \times \cdots \times W_k^{\times a_k}$
  where $n = a_kp^k + \cdots + a_0p^0$ and $W_1 = C_p$ and $W_{i+1} = W_i \wr C_p$.
  As a $W$-equivariant map $f^{\otimes n}$ is an external product of iterated $C_p$-norms of $f$ in the same way.
  This reduces us to proving the following claim:

  \begin{itemize}
  \item If $g$ is $p^m$-torsion in $\Sp^G$, then $g^{\otimes p}$ is $p^{m+1}$-torsion in $\Sp^{G \wr C_p}$.
  \end{itemize}

  For this we make the following computation:
  let $S$ be a $G$-set of cardinality $p^m$ with trivial action.
  In the Burnside ring of $G \wr C_p$ we have
  \[ [S^{\times p}] = p^m[e] + (p^{mp-1} - p^{m-1})[C_p] \]
  where $[e]$ is the trivial $G \wr C_p$-set
  and $[C_p]$ is the $G \wr C_p$-set of order $p$ corresponding to the quotient map $G \wr C_p \to C_p$.
  Norming the equation $0 = p^m g$ up we obtain
  \[ 0 = (p^mg)^{\otimes p} = [S^{\times p}] g^{\otimes p} = \left( p^m + (p^{mp-1} - p^{m-1})[C_p] \right) g^{\otimes p} \]
  Since multiplication by $[C_p]$ factors through the $G^{\times p}$-equivariant category
  we also have $p^m[C_p]g^{\otimes p} = 0$.
  Together these equations imply that $p^{m+1}g^{\otimes p} = 0$ as desired.    
\end{proof}

The reader might wonder whether any of these bounds are close to optimal.
On this point we highlight the following conjecture of Moore.

\begin{cnj}[Moore]
  A simply-connected finite space $X$ has a uniform $p$-torsion bound on its homotopy groups
  if and only if $\pi_*(X) \otimes \Q$ is finite dimensional.
\end{cnj}

Very few cases of this conjecture are known,
and it seems like an interesting question deserving more attention.

%% file: combo.tex
In this appendix we pay our due and prove the combinatorial lemmas we've deferred throughout the paper.
Each lemma involves estimating the asymptotic growth of the rank of some graded algebra over a field.
The proofs are straightforward, if not tedious. 

\begin{lem}[Mahler] \label{lem:steenrod-basic}
  Let $R$ be a graded polynomial algebra with a single generator in each degree of the form $p^n$ for $n \geq 0$.
  Then, there is a constant $C$ such that 
  \[ \left| \log (\mathrm{rank}(R_n)) - \frac{1}{2\log(p)} \log(n)^2 \right| \leq  C\max\left(1,\, \log\log(n) \log(n) \right) \]
\end{lem}

\begin{proof}
  The rank of $R_n$ in degree $n$ can be described as the number of partitions of $n$ into powers of $p$.
  In \cite{MR2921}, Mahler proves the bound\footnote{In fact, the estimate obtained in loc. cit. is substantially sharper than the one described here (his approximation has an error term of order $O(1)$)}
  \[ \log (\mathrm{rank}(R_n)) = \frac{1}{2\log(2)} \log(n)^2 + O(\log\log(n) \log(n)). \]
  Since we will need an estimate valid for small $n$, we have expressed this in a slightly different form.
\end{proof}

The function in the lemma above grows slowly enough that the precise degrees in which the generators are placed isn't particularly important, only their asymptotic distribution.

\begin{cor} \label{cor:steenrod-flexible}
  Given a fixed constant $c$,
  suppose $R$ is a graded polynomial algebra with generators $x_1,x_2,\dots$
  such that the degree of $x_i$ lies in the range $[c^{-1}p^i, cp^i]$.
  Then, we have the following bound on the growth of $R$,
  \[ \log \left( \mathrm{rank}(R_{\leq n}) \right) = \frac{1}{2\log(p)} \log(n)^2 + O(\log\log n \log n). \]
\end{cor}

\begin{proof}

  The key idea is that if we increase the degrees of the generators of $R$, then the rank can only go down. Since we're looking at the cumulative rank there's no need to restrict the grading to only integer degrees. Now let $R'$ denote the algebra where we've dropped the degrees of the generators down to $c^{-1}p^i$ and $R''$ denote the algebra where we've bumped the same degrees up to $cp^i$. 
  Then, we have bounds
  \[ \mathrm{rank}(R_{\leq n}'') \leq  \mathrm{rank}(R_{\leq n}) \leq \mathrm{rank}(R_{\leq n}') . \]
  Up to rescaling the degrees $R'$ and $R''$ are algebras to which we can apply \Cref{lem:steenrod-basic}. To conclude we observe that for a constant $k$ 
  \[ \frac{1}{2\log(p)} \log(kn)^2 + O(\log\log (kn) \log (kn) ) \]  
  simplifies to the desired bound.

\end{proof}

As an example of how the flexibility provided by \Cref{cor:steenrod-flexible} is useful for us,
consider the dual Steenrod algebra.

\begin{exm} \label{exm:steenrod-size}
  At $p=2$ the dual Steenrod algebras is polynomial on classes $\zeta_n$ in degree $2^{n} - 1$.
  This means that with $c=2$ we can treat this like the generators are in $2^{n}$ and \Cref{cor:steenrod-flexible} lets us conclude that
  \[ \log \h^c(\F_2,\, n) =  \frac{1}{2\log(2)} \log(n)^2 + O(\log\log n \log n). \]

  At odd primes things are a little more complicated since the dual steenrod algebra has both the exterior generators $\tau_n$ and the polynomial generators $\xi_n$.
  Using only the $\xi_n$'s we obtain a lower bound on the rank of this algebra.
  For an upper bound we observe that for each $n$ the pair of generators $\tau_n, \xi_{n+1}$ generates an algebra which grows slower than an algebra on a polynomial generator in the degree of $\tau_n$.
  We can then apply \Cref{cor:steenrod-flexible} to this auxiliary algebra.
  Together these bounds give
  \[ \log \h^c(\F_p,\, n) =  \frac{1}{2\log(p)} \log(n)^2 + O(\log\log n \log n). \]
  \tqed
\end{exm}

From here we move into the main body of this appendix where the proofs becomes somewhat more involved.
Before diving into things we give a sequence of remarks which introduce key simplifications.

\begin{rmk} \label{rmk:point-sampling}
  Suppose that $r(n)$ is a non-decreasing function
  (such as the log of the cumulative rank of a graded algebra)
  for which we want to prove that
  \[ r(n) \geq K \log(n)^c + O(\log\log(n) \log(n)^{c-1}). \]

  Then it in fact suffices to prove the bound for a collection of points which have constant logarithmic density. As an example, suppose we know that
  \[ r(p^m) \geq \log(p^m)^c, \]
  then for $ p^m \leq n \leq p^{m+1}$ we have
  \begin{align*}
    r(n)
    &\geq r(p^{m})
    \geq \log(p^{m})^c 
    = \log(p)^cm^c  \\
    &\geq  \log(p)^c(m+1)^c - O(m^{c-1}) 
    \geq  \log(n)^c - O(\log(n)^{c-1}) 
  \end{align*}
  Similar logic applies for upper bounds as well.
  \tqed
\end{rmk}

\begin{rmk} \label{rmk:v2-sampling}
  Expanding on the previous remark,
  the collection of points $m^ap^m$ is sufficiently dense (for constant $a$)
  that the estimate
  \begin{align*}
    \log(  m^a p^m )^c = \log(p^m)^c + O(\log\log(p^m)O(\log(p^m))^{c-1})
  \end{align*}
  allows us to get by with proving the seemingly weaker bound
  \[ r(m^a p^m) \geq \log(p^m)^c + O(\log\log(p^m)O(\log(p^m))^{c-1}). \]
  \tqed
\end{rmk}

\begin{rmk} \label{eqn:tensor-pieces}
  Often the algebras we are working with can be described as a tensor product of several smaller algebras. In bounding these algebras we will use the containments of graded vector spaces
  \[ \left( A^1 \otimes \cdots \otimes  A^h \right)_{\leq n}
    \subseteq A^1_{\leq n} \otimes \cdots \otimes A^h_{\leq n}
    \subseteq \left(A^1 \otimes \cdots \otimes A^h\right)_{\leq hn}. \]
  \tqed
\end{rmk}

\begin{lem} \label{lem:may-basic}
  Let $R$ be the graded polynomial algebra with $n$ generators in each degree of the form $p^n$.
  Then, we have the following bound on the growth of $R$,
  \[ \log (\mathrm{rank}(R_{\leq n})) = \frac{1}{6\log (p)^2} \log(n)^3 + O\left(\log\log(n)\log(n)^2 \right)  \]
\end{lem}

\begin{proof}
  Let $r(n) = \log (\mathrm{rank}(R_{\leq n}))$.
  Describing $R$ as a tensor product of algebras with a single polynomial generator each we can use 
  \Cref{eqn:tensor-pieces} to estimate the growth of $R$.

  \begin{align*}
    r(p^m-1)
    &\leq \sum_{i=0}^{m-1} i \cdot \log(\mathrm{rank}( (k[x_{p^i}])_{\leq p^m-1} )) 
    = \sum_{i=0}^{m-1} i \cdot \log(p^{m-i}) \\
    &= \sum_{i=0}^{m-1} \log(p) i(m-i) 
    = \log(p)\binom{m+1}{3} \\
    &= \frac{1}{6\log(p)^2} \log(p^m)^3 + O(\log(p^m)^2)
  \end{align*}
  Using \Cref{rmk:point-sampling} this suffices to prove the desired upper bound.
  Similarly, for the lower bound we have
  \begin{align*}
    r\left( \binom{m}{2} ( p^m - 1 ) \right)
    \geq \sum_{i=0}^{m-1} i \cdot \log(p^{m-i})
    = \frac{1}{6\log(p)^2} \log(p^m)^3 + O(\log(p^m)^2)    
  \end{align*}
  which by Remarks \ref{rmk:point-sampling} and \ref{rmk:v2-sampling} is sufficient to conclude.

\end{proof}

Again, we observe that the the precise degrees in which the generators are placed isn't critical to the bound in question.

\begin{cor} \label{cor:may-flexible}
  If we allows the degrees of the generators in \Cref{lem:may-basic} to be flexible by a constant factor $c$, then the same bound continues to hold.

\end{cor}


\begin{lem} \label{lem:yn-e2-basic}
  Let $R^h$ be the graded polynomial algebra with $\min(h, n-h+1)$ generators in each degree of the form $p^n$.
  Then, we have the following bound
  \[ \log \left( \max_{h} (\mathrm{rank}(R_{\leq n}^h)) \right) = K_2 \log(n)^3 + O\left(\log\log(n)\log(n)^2 \right) \]
\end{lem}

\begin{proof}
  Let $S^k$ denote the graded polynomial algebra with a generator in each degree of the form $p^n$ for $n \geq k$. Then we can describe $R^h$ as the tensor product
  \[ R^h \cong S^h \otimes \cdots \otimes S^{2h-1}. \]
  Let $r^h(n) = \log( \mathrm{rank}(R_{\leq n}^h) )$,
  let $s(n) = \log( \mathrm{rank}(S^0_{\leq n}) )$,
  and let $r(n) = \max_h(r^h(n))$.
  We estimated the growth of $s(n)$ in \Cref{lem:steenrod-basic}.
  In particular, there exists an error term $\epsilon(n)$ which is monotonically increasing, positive and grows like $O(\log\log(n)\log(n))$ such that
  \[ \left| s(n) - \frac{1}{2\log(p)} \log(n)^2 \right| \leq \epsilon(n). \]
  In order to estimate the growth of $r^h(n)$ in terms of the growth of $s(n)$ we use \Cref{eqn:tensor-pieces} noting that $S^k$ grows according to $s(n/p^k)$.

  We start with the lower bound.
  For each $m \geq 2h-1$ we have
  \begin{align*}
  r^h(h p^m)
    &\geq \sum_{i=h}^{2h-1} s(p^{m-i}) 
      \geq \sum_{i=h}^{2h-1} \left( \frac{1}{2\log(p)} \log(p^{m-i})^2  - \epsilon(p^{m-h}) \right) \\
    &\geq \sum_{i=h}^{2h-1} \left( \frac{\log(p)}{2} (m-i)^2  - \epsilon(p^{m}) \right) \\
    &\geq  \frac{\log(p)}{2} \left( \frac{7}{3} h^3 - 3h^2m + hm^2  - \frac{3}{2}h^2 + hm + \frac{h}{6} \right) - h \epsilon(p^{m}) \\
    &\geq  \frac{\log(p)}{2} \left( \frac{7}{3} h^3 - 3h^2m + hm^2 \right) - \left( \frac{3\log(p)}{4}m^2  + m \epsilon(p^{m}) \right) \\
    &=  \frac{\log(p)}{2} \left( \frac{7}{3} h^3 - 3h^2m + hm^2 \right) - \epsilon_2(p^m)
  \end{align*}
  where $\epsilon_2(n)$ is an error term
  which is independent of $h$ and grows like
  $O(\log\log(n)\log(n)^2)$.
  
  The cubic term in parentheses achieves its maximum (subject to the constraint $m \geq 2h-1$)
  at $h^{\mathrm{max}} = m(3 - \sqrt{2})/7$.
  However, since $h$ must be an integer we instead evaluate at the floor of this number, which we denote $c$.
  Using that the cubic term is monotonically increasing in $h$ when we are just below its local maximum
  we can simplify and obtain the desired bound.  
  \begin{align*}
    r\left( c p^m \right)
    &\geq r^c(c p^m)
    \geq \frac{\log(p)}{2} \left( \frac{7}{3} c^3 - 3c^2 m +  c m^2 \right) - \epsilon_2(p^m)  \\
    &\geq \frac{\log(p)}{2} \left( \frac{7}{3} \left( h^{\mathrm{max}} - 1\right)^3 - 3\left( h^{\mathrm{max}} - 1\right)^2 m +  \left( h^{\mathrm{max}} - 1\right) m^2 \right) + O(\log(m)m^2) \\
    &\geq \frac{(9 + 4\sqrt{2})\log(p)}{294} m^3 + O(\log(m)m^2)
  \end{align*}
  As in the previous lemma Remarks \ref{rmk:point-sampling} and \ref{rmk:v2-sampling} let us conclude.

  Now we prove the upper bound.
  As before we use \Cref{eqn:tensor-pieces} to reduce to $s(n)$.
  For each $m \geq 2h-1 $ we have
  \begin{align*}
  r^h(p^m)
    &\leq \sum_{i=h}^{2h-1} s(p^{m-i}) 
    \leq \sum_{i=h}^{2h-1} \left( \frac{1}{2\log(p)} \log(p^{m-i})^2  + \epsilon(p^{m-h}) \right) \\
    &\leq  \frac{\log(p)}{2} \left( \frac{7}{3} h^3 - 3h^2m + hm^2 \right) + \epsilon_3(p^m)
  \end{align*}
  where $\epsilon_3(n)$ is much like $\epsilon_2(n)$. 
  This is the same cubic term as before.
  Evaluating at its maximum 
  (recall again that $m \geq 2h-1$)
  we obtain a suitable bound on $r^h(p^m)$
  \begin{align*}
    r^h(p^m) \geq \frac{(9+4\sqrt{2})\log(p)}{294} m^3 + O(\log(m)m^2).
  \end{align*}
  
  Examining the tensor product decomposition of $R^h$ we see that $R^{h-1}$ contains every polynomial generator $R^h$ contains until the beginning of $S^{2h-1}$ in degree $p^{2h-1}$.
  This means that we can drop the terms in the maximum where $h$ is too large.
  \begin{align*}
    r(p^m)
    &\leq \max_{h} r^h(p^m)
      = \max_{m \geq 2h-1} r^h(p^m) 
      \leq \frac{(9+4\sqrt{2})\log(p)}{294} m^3 + O(\log(m)m^2)
  \end{align*}
  Again the numbers of the form $p^m$ are sufficiently dense that \Cref{rmk:point-sampling} lets us conclude.
  \qedhere

\end{proof}

\begin{cor} \label{cor:yn-e2-flexible}
  If we allows the degrees of the generators in \Cref{lem:yn-e2-basic} to be flexible by a constant factor $c$, then the same bound continues to hold.
\end{cor}

\begin{lem} \label{lem:yn-einf-basic}
  Let $R^h$ be the graded polynomial algebra with $\binom{h}{2}$ polynomial generators distributed so that there is one generator in degree $p^{h+1}$, two generators in degree $p^{h+2}$, three generators in degree $p^{h+3}$, etc. Then, as $h$ varies we have the following bound
  \[ \log \left( \max_h ( \mathrm{rank} ( R^h_{\leq n} ) ) \right) = \frac{2}{75 \log(p)^2} \log(n)^3 + O(\log\log(n)\log(n)^2).\]    
\end{lem}

\begin{proof}
  Let $r^h(n) = \log( \mathrm{rank}(R_{\leq n}^h) )$
  and let $r(n) = \max_h(r^h(n))$.
  As in the previous lemma we use \Cref{eqn:tensor-pieces} to obtain
  upper and lowers bounds on the size of $r^h$.
  This time we think of $R^h$ as a tensor product of algebras on a single polynomial generator each.

  We start with the upper bound.
  Using the fact that $R^h_{\leq p^m-1}$ has only the unit for $h \geq m-1$ we have 
  \begin{align*}
    r(p^m -1)
    &\leq \max_h r^h(p^m -1)
      \leq \max_{h < m} r^h(p^m - 1)
      \leq \max_{h < m}\left( \sum_{i=1}^{h-1} i \cdot \log \left( p^{m-h-i} \right) \right) \\
    &= \max_{h < m} \frac{\log(p)}{6}\left( 3h^2m - 5h^3 + 6h^2 - 3hm - h \right) \\
    &\leq \max_{h < m}  \frac{\log(p)}{6} \left( 3h^2m - 5h^3 + 6m^2 \right) \\
    &\leq \frac{\log(p)}{6} \left( 3\left( \frac{2}{5} m \right)^2m - 5\left( \frac{2}{5} m \right)^3 + 6m^2 \right) \\
    &\leq \frac{2\log(p)}{75}m^3 + \log(p)m^2
  \end{align*}
  which by \Cref{rmk:point-sampling} is sufficient to conclude.
  
        
  

  The lower bound is proved following the same pattern as the previous lemma.
  Again restricting to $ h < m $ we have 
  \begin{align*}
    r^h\left( \binom{h}{2} p^m \right)
    &\geq \sum_{i=1}^{h-1} i \cdot \log ( p^{m-h-i} )       
      = \frac{\log(p)}{6} \left( 3h^2m - 5h^3 + 6h^2 - 3hm - h \right) \\
    &\geq \frac{\log(p)}{6} \left( 3h^2m - 5h^3  \right) + O(m^2).   
  \end{align*}
  Evaluating at $h = \left\lfloor 2m/5 \right\rfloor$, simplifying and
  using Remarks \ref{rmk:point-sampling} and \ref{rmk:v2-sampling}
  then gives the desired conclusion.
  \qedhere

\end{proof}
  
\begin{cor} \label{cor:yn-einf-flexible}
  If we allows the degrees of the generators in \Cref{lem:yn-e2-basic} to be flexible by a constant factor $c$, then the same bound continues to hold.

\end{cor}




